\documentclass[11pt]{elsarticle}

\makeatletter
\def\ps@pprintTitle{%
	\let\@oddhead\@empty
	\let\@evenhead\@empty
	\def\@oddfoot{}%
	\let\@evenfoot\@oddfoot}
\makeatother

\usepackage{amsfonts,amssymb,amsmath,amsthm}
\usepackage{comment}
\usepackage{ulem}
\usepackage[margin=1in]{geometry}
\usepackage{setspace}
\usepackage{caption} 
\usepackage{subcaption}
\usepackage{tikz}
\usepackage{parskip}
\usetikzlibrary{mindmap,shadows}
\usepackage[justification=centering]{caption}
\PassOptionsToPackage{hyphens}{url}

\usepackage{amsthm,amsmath,amsfonts,amssymb,bm}
\usepackage[ruled,vlined]{algorithm2e}
\usepackage[flushleft]{threeparttable} 

\usepackage{graphicx,psfrag,epsf}
\usepackage{multicol}
\usepackage{multirow}
\usepackage{color,hyperref,xcolor}
\hypersetup{colorlinks=true,urlcolor=blue,citecolor=purple}
\usepackage{longtable}
\usepackage{natbib}
\usepackage{lineno,hyperref}
\usepackage{cleveref}
\usepackage{float}

\numberwithin{equation}{section}

\newtheorem{definition}{Definition}
\newtheorem{lemma}{Lemma}
\newtheorem{prop}{Proposition}
\newtheorem{theorem}{Theorem}

\newtheorem{corollary}{Corollary}

\newtheorem{remark}{Remark}

\def\Z{\mathbb{Z}}
\def\N{\mathbb{N}}
\def\E{\mathbb{E}}
\def\P{P}

\def\A{\mathcal{A}}
\def\tA{\mathcal{\tilde{A}}}

\numberwithin{theorem}{section}
\numberwithin{lemma}{section}
\numberwithin{corollary}{section}
\numberwithin{definition}{section}
\numberwithin{result}{section}
\numberwithin{remark}{section}
\numberwithin{assumption}{section}
\numberwithin{prop}{section}
\newcommand{\update}[1]{{\color{black} #1}}

\begin{document}
	
	\begin{frontmatter}

		\title{How fast do rumours spread?}
		
		\author[UoE]{Rishideep Roy}
		\author[AU]{Kumarjit Saha}

		\affiliation[UoE]{organization={School of Mathematics, Statistics and Actuarial Science, University of Essex},
			addressline={Wivenhoe Park}, 
			city={Colchester},
			postcode={CO4 3SQ}, 
			state={Essex},
			country={UK}}
		
		\affiliation[AU]{organization={Department of Mathematics, Ashoka University},
			addressline={NH 44, Rajiv Gandhi Education City}, 
			city={Sonepat},
			postcode={131029}, 
			state={HR},
			country={India}}

		\begin{abstract}
			We study a rumour propagation model along the lines of \cite{lebensztayn2008disk} as a long-range percolation model on $\Z$. We begin by showing a sharp phase transition-type behaviour in the sense of exponential decay of the survival time of the rumour cluster in the sub-critical phase.
			
			In the super-critical phase, \update{under the assumption that radius of influence r.v. has $2+\epsilon$ moment finite (for some $\epsilon>0$)}, we show that the rightmost vertex in the rumour cluster has a deterministic speed in the sense that after appropriate scaling, the location of the rightmost vertex converges a.s.\ to a deterministic positive constant. \update{Under the assumption that radius of influence r.v. has $4+\epsilon$ moment finite,} we obtain a central limit theorem for appropriately scaled and centered rightmost vertex. 
			Later, we introduce a rumour propagation model with reactivation. 
			\update{For this section, we work with a family of exponentially decaying i.i.d. radius of influence r.v.'s, and we obtain the speed result 
				for the scaled rightmost position of the rumour cluster. Each of these results is novel, in the sense that such properties have never been established before in the context of the rumour propagation model on $\Z$, to the best of our knowledge.}

		\end{abstract}
		
		
	\end{frontmatter}

	\section{Introduction} 
	Interacting particle systems have been an area of interest for a long time. There have been numerous works related to varied forms of it. We are particularly interested in the spread of rumours within networks. 
	We work with the rumour propagation model \update{having} rumours propagating on an infinite graph, which is within the class of long-range percolation models. We study rumour percolation models and we are interested sharp phase transition (in terms of the survival time of the rumour cluster) and speed of the scaled boundary points of the rumour cluster. We want to highlight here that extinction probabilities and phase transitions results have been studied on rumour models for some time. Still, to our knowledge, there has not been any previous work on the speed of propagation of rumours, which is our novelty. In addition, we also obtain a CLT result for the rightmost particle of the rumour cluster. In our work, we further generalise the rumour model to allow for reactivations, which
	basically means that a particle 
	who heard the rumour in the past may spread it again at a future reactivation time point. 
	To the best of our knowledge, this model has never been studied before. \update{The concept of reactivation is influenced by the idea of recurring public opinion and its dissemination, considered for example in \cite{xu2020dynamic}. It is further motivated by the concept of reinfection in case of disease spread models, which have been studied in works such as \cite{yortsos2023model}.}

	There are close resemblances between rumour propagation and epidemic models as observed in \cite{berg1983}. However, they have their differences also, as stated in \cite{daley1965} that the fraction of the population who ultimately hears a rumour is independent of the population size. Both of them consider a finite population. Daley in \cite{daley1967} discussed stochastic and deterministic versions of the rumour process within a finite population. Dunstan in \cite{dunstan1982rumour} extended the work on these models. 
	There have been other works too on the spread of rumours in a finite population, for example, in \cite{frieze1985}. They consider a complete graph with $n$ nodes, and the main object of interest is the time required for the rumour to spread to the entire network. This model assumes that the ones hearing the rumour will always spread it to other vertices that have never heard it before. Pittel in \cite{pittel1987} extends results on this model. 
	
	Lebensztayn et al.\ in \cite{lebensztayn2008disk} consider a long-range percolation model, which they call a disk-percolation model, for the spread of infection. The model can be invariably used for rumour propagation as well. They have random variables termed radii of infection linked to each vertex $v$ in a locally finite and connected graph $\mathcal{G}$, which follow geometric distributions. Initially, only the root is infected, and with time, the infection grows. They study the question of the indefinite spread of infection with positive probability and critical probability based on the parameter of the geometric distribution for the same. Junior et al.\ in \cite{junior2011} consider a similar model on $\N$, which is called the Firework process. There, the propagation of rumours is uni-directional. They also consider a similar process called the reverse firework process, where instead of the propagation of rumours by an active vertex, the rumour is collected by the inactive vertices based on the random variables associated with inactive vertices. Their results consider generalisations over the distributions of the random variables associated with the vertices. Gallo et al.\ in \cite{GGJR2014} use the same models and construct a renewal process associated with the rumours to arrive at their results. Alsadat et al.\ in \cite{alsadat2019rumour} consider a variant of the firework process, where rumours are spread to a new vertex only if it receives them from at least two sources. Junior et al.\ in a survey in  \cite{junior2019rumor} consider various variants of the rumour model and the major works done on them. Some recent works on the rumour model include \cite{esmaeeli2020rumour}, \cite{esmaeeli2023probability}. \update{For a long range percolation model the question of eventual coverage of $[0, \infty)^d$ for $d \geq 1$ has been studied in \cite{athreya2004}. However the question of rumour percolation on $\Z$ (which essentially means complete coverage) is different than the question of eventual coverage.}
	
	The shape of percolation clusters has also been studied for a long period. On $\Z$, the study of the shape of the percolation cluster is equivalent to the study of the speed of propagation of the cluster. In the case of contact processes on $\Z$, another closely related stochastic growth model, it has been shown in \cite{durrett1980growth} that the rightmost particle grows at a linear rate, starting from one active particle at time zero at the origin. In the case of oriented percolation on $\Z^2$, a similar result has been shown in \cite{durrett84}. For generalised oriented site percolation on $\Z^2$, it has been shown that the rightmost infinite open path has a linear speed in \cite{hartarsky2021generalised}. Co-existence in interacting particle systems is another question of interest in the literature. Kostka et al.\ in  \cite{kostka2008word} study competing rumours in social networks. Co-existence in another related percolation model, the frog model, is studied in \cite{deijfen2019}, \cite{roy2021}. 
	
	Below, we describe our \update{initial} rumour propagation model. 
	On the infinite graph $\Z$ we consider an i.i.d. family of non-negative random variables such that these random variables are linked to every vertex $z \in \Z$. We term these variables as radius of influence random variable. A rumour is initiated in one of these vertices, and we are interested in the phenomenon of the indefinite spread of the rumour. In the beginning, we assume that the radius of influence random variable attached to a vertex, does not change over time. As a result, after hearing the rumour, a vertex can spread the rumour only once and that too only at the time of hearing. \update{ Under the assumption that the $1+\epsilon$ moment of radius of influence random variable is finite (for some $\epsilon > 0$), we establish a sharp phase transition behaviour in rumour percolation in terms of sharp decay of the survival time of the subcritical rumour cluster (see \Cref{thm:phase}). Our methods can be successfully extended to analyze some rumour propagation models where individuals are randomly placed on $\Z$ (see Remark \ref{rem:Extensions} for more details).}
	
	The next aspect we study is the speed of growth of the rumour cluster in the supercritical phase. 
	\update{ In \Cref{thm:rightmost-speed}, we show that in the supercritical phase, under the assumption that the $2+\epsilon$ moment of the radius of influence random variable is finite for some $\epsilon > 0$, the appropriately scaled rightmost vertex in the rumour cluster converges to a deterministic positive constant almost surely.}
	This result can be interpreted as the `speed' of the growth of the rumour cluster in the supercritical phase. \update{In \Cref{thm:CLT} under the assumption that the $4+\epsilon$ moment is finite, we obtain the central limit theorem for the appropriately scaled and centred rightmost vertex in the rumour cluster.}
	
	Later, we generalise our model and allow vertices that heard the rumour in the past to spread it at multiple re-activation time points using independent copies of radii of influence random variables. Even though re-activation is a phenomenon that has been considered in different interacting particle systems, to the best of our knowledge, it has not been considered in the rumour propagation model so far. This provides a 
	greater flexibility to individual vertices to spread the rumour. 
	As a result, the question of phase transition in terms of rumour percolation becomes trivial. For this re-activation model, we deal with an i.i.d. family of radius of influence random variables with exponentially decaying tails and we show that the appropriately scaled rightmost vertex in the rumour cluster almost surely converges to a deterministic positive constant as well (see \Cref{thm:restarst_rightmost-speed}). 
	
	For both models, our proof techniques rely heavily on a renewal construction (see \cref{subsec:RenewalConstruction}). Later, in  
	\cref{subsec:Renewal_restart} we modify this renewal construction in a non-trivial manner 
	for the  model with re-activation. The differences between the rightmost vertices of rumour clusters between successive renewals, together with the time gaps between successive renewal steps are shown to be i.i.d. 

	The paper is organised as follows.
	The model we work with initially, is formally defined in \Cref{sec:model}. We exhibit sharpness of phase transition regarding rumour percolation in \Cref{sec:righttail}. The result related to the ‘speed’ of growth of the rumour cluster in the supercritical phase is proved in \Cref{sec:speedrttail}. In the same section we obtain a central limit theorem for suitably centred and scaled position of the right-most vertex of the rumour cluster. In \Cref{sec:reactivation} as a generalisation of the usual rumour model, we introduce a rumour 
	propagation model with independent re-activation clocks attached to each vertex and we obtain a deterministic speed for the growth of the rumour cluster in this set-up. The renewal constructions, especially in the case of the re-activated rumour model, are non-trivial.

	We shall make use of the following lemma in our proofs. This lemma is a standard result of analysis, and as such, the proof is left to the readers. 
	
	\begin{lemma}\label{lem-infprod}
		For a sequence of real numbers $q_n$, where $0 \leqslant q_n <1$, the infinite product $\prod_{n=1}^{\infty} (1-q_n)$ converges to a non-zero number iff $\sum_{n=1}^{\infty} q_n$ converges. 
	\end{lemma}
	
	\section{Rumour model}\label{sec:model}
	We define our model of rumour percolation following the \update{lines of \cite{lebensztayn2008disk}, which is also discussed in \cite[Section 3]{junior2019rumor}}.  
	We consider the one-dimensional integer lattice $\mathbb{Z}$ as our intended network on 
	which we wish to study the percolation of a rumour. We denote the set $\Z\cap [0, \infty)$ by $\Z^+$ and $\Z\cap (-\infty, 0]$ by $\Z^-$. We consider 
	$\{ I_z : z \in \Z \}$, a collection of independent and identically distributed (i.i.d.) 
	non-negative random variables, where $I_z$ represents the radius of influence for rumour propagation at the vertex $z \in \Z$. 
	\update{Let $p_1 := \P(I_z \in [0,1)) $ and we assume that $\E(I_z^{1 + \epsilon}) <\infty$ for some $\epsilon > 0$.} 
	Using this collection $\{I_z : z \in \Z\}$, we define the rumour percolation process as follows: 
	\begin{itemize}
		\item[1.] At time $n=0$, only a vertex $x \in \Z$ hears a rumour, becomes active
		and spreads the rumour to all vertices within $I_x$ distance from it.
		The rest of the vertices remain inactive (in terms of propagating the rumour). 
		Without loss of generality, we can assume $x$ to be $0$. 
		Set $\A_0 = \{0\}$, which denotes the set of vertices that have heard a rumour at $0$. 
		
		\item[2.] The set of vertices that have heard the rumour by time $n=1$ is given by 
		$$
		\A_1 := \{z : |z| \leqslant I_0\}.
		$$  
		Vertices in $\A_1$ spread the rumour to new vertices which have not heard the 
		rumour till then, according to their respective radii of influences. 
		Precisely, $z \in \A_1$ spreads rumour to an inactive vertex $u$ in the set 
		$\Z \setminus \A_1 $ at time $2$ if and only if $|u - z|\leqslant I_z$. 
		
		We note that the vertex $0 \in \A_1$, that has heard the rumour in the past, 
		fails to spread the rumour to any new vertex after time $1$.
		The set of newly activated vertices at time $1$, denoted by $\tilde{\A}_1 
		:= \A_1 \setminus \A_0$ may spread the rumour to other inactive vertices that 
		have not heard the rumour till then, and 
		make them active for the next time point. Keeping this in mind, 
		the vertex $0$ becomes effectively inactive from time $1$ onward as it cannot spread the rumour to any new vertex after time $1$ and 
		$\tilde{\A}_1 := \A_1 \setminus \A_0 = \A_1 \setminus  \{ 0\}$ denotes the set of 
		(effectively) active vertices at time $1$. 
		
		Similarly, the set of active vertices at time $2$, which have heard the rumour at time $2$ and can propagate the rumour to new inactive vertices at time $3$, is given by 
		$$
		\tilde{\A}_2 := \{z \in \Z \setminus \A_1 : 
		|z - u| \leqslant I_u \text{ for some }u \in \tilde{\A}_1\}.
		$$
		The set of vertices that have heard the rumour  by time $2$ is given by 
		$$
		\A_2 := \A_1 \cup \tilde{\A}_2.
		$$
		
		\item[3.] More generally, at time $n \geqslant1$, let $\tilde{\A}_n$ denote the set of active 
		vertices at time $n$ and the set of vertices that have heard the rumour by time $n$,
		is given by $\A_n :=  \A_{n-1} \cup \tilde{\A}_n$. An active vertex
		$u \in \tilde{\A}_n$ spreads rumour to all inactive vertices within $I_u$ distance 
		that never heard the rumour in the past and 
		makes them active at time $n+1$. We observe that vertices 
		in the set $\A_{n-1}$ can not spread the rumour to any new inactive vertex 
		at time $n$ or later. Consequently, the set of 
		active vertices at time $n+1$ is given by 
		$$
		\tA_{n+1} := \{z \in \Z \setminus \A_n : |z - u| \leqslant I_u \text{ for some }u \in \tilde{\A}_n\}, 
		$$
		and $\A_{n+1} := \A_{n} \cup \tilde{\A}_{n+1}$ denotes the set of 
		vertices which have heard the rumour by time $n+1$.
	\end{itemize}
	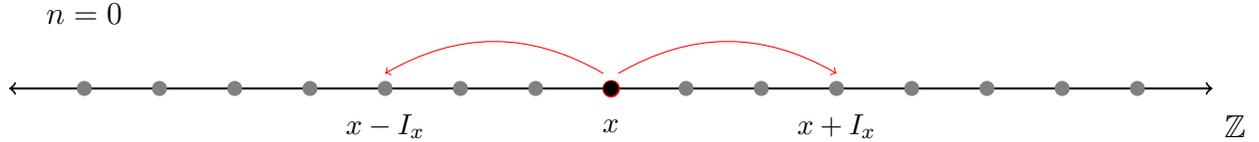
\begin{figure}[!h]
		\begin{center}
			\begin{tikzpicture}
				\draw[thick,<->] (-8,0) -- (8,0);
				\foreach \x in {-7, ..., 7} {
					\path[fill=gray] (\x, 0) circle[radius=1mm];
				}
				\node[] at (8.3,-0.5) {\large $\Z$};
				\path[fill=black,circular glow={fill=red}] (0,0) circle[radius=1mm];
				\node[] at (0,-0.5) { $x$};
				\node[] at (-7,1) {\large $n=0$};
				\draw [red, ->] (-0.1,0.2) to [out=150,in=30,looseness=1] (-3,0.2);
				\node[] at (-3,-0.5) { $x-I_x$};
				\draw [red, ->] (0.1,0.2) to [out=30,in=150,looseness=1] (3,0.2);
				\node[] at (3,-0.5) { $x+I_x$};
			\end{tikzpicture}
		\end{center}
		\caption{Rumour percolation at time $n=0$, with rumour starting at $x \in \Z$}
		\label{fig:rum_t0}
	\end{figure}
	
	\Cref{fig:rum_t0} and \Cref{fig:rum_t1} together give a pictorial representation of the rumour process, starting from $x \in \Z$. Grey dots denote the inactive vertices which have never heard a rumour. Black dots represent vertices that have heard the rumour so far, and black dots with red boundaries denote vertices that have heard the rumour newly, i.e., the set of currently active vertices. In the figure, we can see that  $\tilde{\A}_{0}=\{x\}$ and $I_x=3$. Hence, $\A_1=\{x-3,\ldots, x+3\}$ and $\tilde{\A}_{1}=\A_1 \backslash \{0\}$, as is shown in \Cref{fig:rum_t1}. Assuming $I_{x+1}=2$, $I_{x+2}=3$, $I_{x+3}=1$, $I_{x-1}=3$, $I_{x-2}=2$ and $I_{x-3}=4$, we 
	have $\A_2 = \{x-7, \ldots, x+5 \}$ and $\tA_2 = \{x-7, \ldots, x+5 \} \setminus \A_1$. 
	In this way, the rumour process evolves. 
	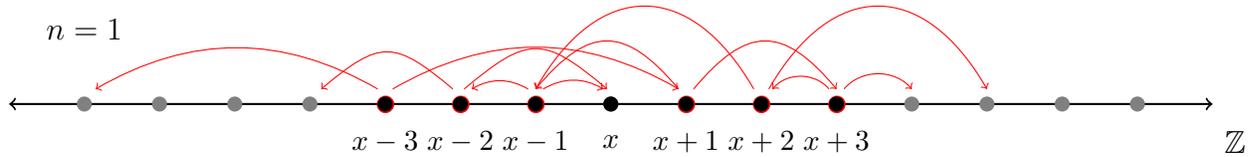
\begin{figure}[!h]
		\begin{center}
			\begin{tikzpicture}
				\draw[thick,<->] (-8,0) -- (8,0);
				\foreach \x in {-7, ..., 7} {
					\path[fill=gray] (\x, 0) circle[radius=1mm];
				}
				\node[] at (8.3,-0.5) {\large $\Z$};
				\foreach \x in {-3,-2,-1,1,2,3} {
					\path[fill=black,circular glow={fill=red}] (\x,0) circle[radius=1mm];}
				\path[fill=black] (0, 0) circle[radius=1mm];
				\node[] at (0,-0.5) { $x$};
				\node[] at (-7,1) {\large $n=1$};
				\draw [red, ->] (.9,0.2) to [out=130,in=50,looseness=1.5] (-1,0.2);
				\draw [red, ->] (1.1,0.2) to [out=50,in=130,looseness=1.5] (3,0.2);
				\draw [red, ->] (2.1,0.2) to [out=60,in=120,looseness=1.5] (5,0.2);
				\draw [red, ->] (1.9,0.2) to [out=120,in=60,looseness=1.5] (-1,0.2);
				\draw [red, ->] (2.9,0.2) to [out=130,in=50,looseness=1] (2.15,0.2);
				\draw [red, ->] (3.1,0.2) to [out=50,in=130,looseness=1] (4,0.2);
				\draw [red, ->] (-2.9,0.2) to [out=30,in=150,looseness=1] (0.9,0.2);
				\draw [red, ->] (-3.1,0.2) to [out=150,in=30,looseness=1] (-6.85,0.2);
				\draw [red, ->] (-0.9,0.2) to [out=30,in=150,looseness=1] (-0.1,0.2);
				\draw [red, ->] (-1.1,0.2) to [out=150,in=30,looseness=1] (-1.85,0.2);
				\draw [red, ->] (-2.1,0.2) to [out=140,in=40,looseness=1.5] (-3.85,0.2);
				\draw [red, ->] (-1.95,0.2) to [out=40,in=140,looseness=1.5] (-0.05,0.2);
				\foreach \y in {1,...,3}
				\node[] at (\y,-0.5) { $x+\y$};
				\foreach \y in {1,...,3}
				\node[] at (-\y,-0.5) { $x-\y$};
			\end{tikzpicture}
		\end{center}
		\caption{Rumour percolation at time $n=1$, with rumour starting at $x \in \Z$ at time $n=0$}
		\label{fig:rum_t1}
	\end{figure}
	
	The set $\A_n$ represents the rumour cluster at time $n$, i.e., the 
	collection of vertices that have heard the rumour by time $n$. 
	We observe that for our model and for all $n \geqslant 1$
	the rumour cluster $\A_n$ is a finite set of the form $[a, b] \cap \Z$
	for some  $a, b \in \Z$ with $a \leqslant 0 \leqslant b$. 
	In what follows, for any set $A$, the notation $\# A$ denotes the cardinality of the set. 
	\begin{definition} 
		\label{def:RumourPercolation}
		The rumour percolation event is defined as 
		$\# \A_n \to \infty$ as $n \to \infty$.
	\end{definition}
	Clearly, for this rumour propagation model, we have rumour percolation iff 
	$\tilde{\A}_{n} \neq \emptyset \text{ for all }n$.
	In other words, rumour percolates if the set of active vertices never becomes an empty set. 
	Otherwise, rumour stops percolating. It is \update{also easy to observe} that if $\tA_{n_0} = \emptyset$ 
	for some $n_0 \in \N$, then $\tA_n = \emptyset$ for all $n \geqslant n_0$. \update{This, however does not necessarily hold for the rumour propagation model with re-activation, introduced in \Cref{sec:reactivation}.}
	In the next section, we present a sharp phase transition result 
	for rumour percolation.


	\section{ Sharp phase transition for Rumour process}\label{sec:righttail}

	In this section, we show that the rumour percolation process exhibits a sharp phase transition. \update{For $n \in \N$ we define the event $B_n$ as
		$$
		B_n := \{ \text{either }n \in \A_m \text{ or }-n \in \A_m \text{ for some }m\geq 1\}.
		$$
		Starting with the set $\A_0 = \{0\}$ the rumour percolation event is defined as the event $\cap_{n = 1}^\infty B_n$. The rumour percolation probability is denoted by $\gamma$. We observe that for the present model $\gamma$ can be simply expressed as
		$$
		\gamma  = 
		\P(\cap_{n = 1}^\infty \bigl ( \tilde{\A}_n \neq \emptyset ) \bigr ). 
		$$
		However, for a rumour propagation model with re-activation, the above expression of $\gamma$ does not hold.}
	
	Let $\tau$ denote the random time when the rumour stops  to percolate, i.e., 
	\begin{align*}
		\tau := \inf\{ n \geqslant 1 : \tilde{\A}_m = \emptyset \text{ for all }m \geq n\},
	\end{align*} 
	and for this model we have
	$\tau = \inf\{ n \geqslant 1 : \tilde{\A}_n = \emptyset \}$.
	The rumour percolates if the r.v.\ 
	$ \tau $ takes the value $ + \infty$. 
	\update{The next proposition gives us a necessary and sufficient condition for positive probability of rumour percolation.
		
		\begin{prop}
			\label{prop:phase_transition}
			For $n \geq 1$ we define $a_n := \prod_{i=0}^{n} P(I_0 \leq i)$. Then 
			$$
			\gamma = 0 \text{ if and only if }\sum_{n=1}^\infty  a_n  = \infty.
			$$
		\end{prop}
		Proposition \ref{prop:phase_transition} is an extension of Theorem 2.1 of  \cite{junior2011} where uni-directional rumour percolation process on $\N$ was studied. 
		\begin{proof}
			For $n \geq 1$ we define the events 
			\begin{align*}
				B^+_n & := \{ u + I_u < n+1 \text{ for all }u \in [0,n]\cap \Z \} \text{ and }\\ 
				B^-_n & := \{u - I_u > - n - 1 \text{ for all }u \in [-n, 0]\cap \Z \}.
			\end{align*}
			From the translation invariance nature of our model it follows that $\P(B^+_n) = \P(B^-_n) = \prod_{i=0}^{n} P(I_0 \leq i) = a_n$. We define the random variables $J^+, J^-$ respectively as 
			\begin{align*}
				J^+ & := \sup\{ n \geq 1 : \text{ the event }B^+_n \text{ occurs} \} \text{ and }\\ 
				J^- & := \sup\{ n \geq 1 : \text{ the event }B^-_n \text{ occurs} \}.
			\end{align*}
			If $\sum_{n=1}^\infty a_n < \infty$, then by Borel-Cantelli lemma we have that $\P(J^+ < \infty) = \P(J^-  < \infty) =1$ which gives us $\P(J^+\vee J^- < \infty) =1$ as well. Choose $M \in \N$ such that $\P(J^+\vee J^- \leq M) > 0$, where $'\vee'$ denotes the maximum of two numbers. 
			We observe that the event $\{ J^+\vee J^- \leq M \}$ depends only on the family $\{ I_u : u \in \Z, u \notin [-M, M] \}$ and on this event this collection 
			$\{ I_u : u \in \Z, u \notin [-M, M] \}$ is such that if the rumour starts propagating from the vertex $M+1$ (-$M+1$) towards the right (left) then it never stops. Because if it stops, that would indicate presence of $M^\prime > M$ such that the event $B^+_{M^\prime}\cup B^-_{M^\prime}$ occurs. This gives a contradiction as such a $M^\prime$ should not exist on the event $\{ J^+\vee J^- \leq M \}$. Now we choose the collection $\{ I_u : u \in \Z, u \in [-M, M] \}$ such that the rumour started from $0$ reaches the vertex $M+1$ and we can always do that with positive probability. This gives us that $\gamma > 0$ if $\sum_{n=1}^\infty a_n < \infty$.  
			
			To complete the proof of Proposition \ref{prop:phase_transition} we need to show that $\gamma = 0$ if $\sum_{n=1}^\infty a_n = \infty$.
			To do that we define an auxiliary construction of the rumour propagation process using i.i.d. collection $\{I^{(n)}_z : z \in \Z, n \in \N\}$ such that for each $n \geq 1$, a vertex $u \in \tA_n$ uses the r.v. $I^{(n+1)}_u$ to propagate the rumour. We observe that the process obtained through this auxiliary construction has the same distribution as the original one. Based this auxiliary construction we define the corresponding versions of the events $B_n^+, B_n^-$ respectively as  
			\begin{align*}
				B^+_n & := \{ u + I^{(n+1)}_u < n+1 \text{ for all }u \in [0,n]\cap \Z \} \text{ and }\\ 
				B^-_n & := \{u - I^{(n+1)}_u > - n - 1 \text{ for all }u \in [-n, 0]\cap \Z \}.
			\end{align*}
			Clearly we have $\P(B^+_n) = \P(B^-_n) = a_n$ and we also have that the events $B^+_n$ for $n \geq 1$ are mutually independent. Therefore, the converse part of Proposition \ref{prop:phase_transition} also follows from Borel-Cantelli lemma. This completes the proof.
		\end{proof}
	}
	
	In addition, if we assume that $\E(I_0^{1+\epsilon}) < \infty$ for some $\epsilon > 0$ then 
	\Cref{thm:phase} shows that there is a sharp phase transition
	for rumour percolation probability.
	
	\begin{prop}
		\label{thm:phase}
		\update{If $\E(I_0^{1+\epsilon}) < \infty$ for some $\epsilon > 0$ then we have 
			\begin{align*}
				\gamma = 
				\begin{cases}
					0 & \text{ if } p_1 > 0\\
					1 & \text{ if } p_1 = 0.
				\end{cases}
		\end{align*}}
		Further, for $p_1 > 0$ there exists $C_0, C_1 > 0$ depending only on $p_1$
		such that for all $n \in \N$ we have 
		\begin{align}
			\label{eq:Tau_Exp}
			\P(\tau > n) \leqslant C_0 e^{(- C_1 n)}.
		\end{align} 
	\end{prop}
	We mention here that \cite[Remark 5]{lebensztayn2008disk} 
	showed that rumour percolation on $\Z$ doesn't happen when $p_1 > 0$. \update{The authors of \cite{lebensztayn2008disk} considered rumour percolation model only for geometrically distributed radii of influence random variables but on more general networks like $\Z^d$ for $d\geq 2$ and on regular trees.} 
	A proof of this can be found from \cite[Theorem 2.1]{LMM08}. 
	It is important to observe that in addition to the phase transition result for rumour percolation, as in \cite[Theorem 2.1]{LMM08}, \Cref{thm:phase} gives a sharp 
	phase transition result \update{in the sense of exponential decay of the amount of time the sub-critical rumour cluster survives. In Remark \ref{rem:Extensions} we will further, apply our method for some other variants of rumour propagation model on $\Z$ where individuals are randomly placed over sites of $\Z$. These models were earlier proposed in \cite[Section 5]{GGJR2014} as possible extensions to be explored. 
		Later we will show that, under the assumption of exponentially decaying i.i.d. radius of influence random variables, this means that the size of the sub-critical rumour cluster decays exponentially.} 
	Before we proceed further, we mention that several qualitative 
	results of this paper involve constants. For more clarity, 
	we will use $C_0$ and $C_1$ to denote two positive constants 
	whose exact values may change from one line to another. 
	The important thing is that both $C_0$ and $C_1$ are universal constants 
	whose values depend only on the parameters of the process.
	
	In order to prove \Cref{thm:phase}, we define an `overshoot' r.v.\
	based on the collection $\{ I_{i}:i \in \Z^- \}$
	as follows:   
	\begin{align}
		\label{def:J_Geometric_Tail}
		O := \sup\{ i + I_{i} : i \in \Z^- \}.
	\end{align}
	Clearly, the overshoot r.v.\ $O$ is non-negative a.s.\ and it 
	represents the amount of overshoot on $\Z^+$, the positive part 
	of the integer line, caused due to the family of r.v.'s $\{I_{i}:i \in \Z^- \}$. 
	\update{Simple application of Borel Cantelli lemma shows that under the assumption of $(1+\epsilon)$ finite moment assumption,  the overshoot r.v. is 
		finite a.s.\ \Cref{lem:rumour-block} analyzes some properties of this overshoot r.v.  
		\begin{lemma}
			\label{lem:rumour-block}
			\begin{itemize}
				\item[(i)] For any $p_1 > 0$ we have $\P( O = 0) > p^\prime $, where $p^\prime >0$; 
				\item[(ii)] If $\P(I_0 \leqslant 1) > 0$ then we have $\P( O \leqslant 1) > 0$;
				
				\item[(iii)] 
				\begin{itemize}
					\item[(a)] If $\E(I_0)^{2 + \epsilon} < \infty$ for some $\epsilon > 0$ then we have $\E(O) < \infty$.
					\item[(b)] If $\E(I_0)^{4 + \epsilon} < \infty$ for some $\epsilon > 0$ then we have $\E(O^2) < \infty$.
					\item[(c)] If $I_0$ has exponentially decaying tails then there exist constants $C_0, C_1 > 0$  such that 
					$$
					\P( O > n) \leqslant C_0 e^{-C_1 n}\text{ for all }n \in \N.
					$$ 
				\end{itemize}
			\end{itemize}
	\end{lemma}}
	\update{\begin{proof}
			For (i), using Markov inequality we obtain 
			\begin{align*}
				\P(O  = 0) & = P(\cap_{i=0}^{\infty} (I_{-i} \leq i)) \\
				& =  p_1\prod_{i=1}^{\infty}
				P(I_{-i} \leq i) \\
				&   =  p_1\prod_{i=1}^{\infty}
				( 1 - P(I_{-i} > i)) \\
				& \geq  p_1\prod_{i=1}^{i_0}( 1 - P(I_{-i} > i))\prod_{i=i_0+1}^{\infty}
				( 1 - \E(I_{-i}^{1 + \epsilon})/i^{1+\epsilon}) \\
				& =  p_1\prod_{i=1}^{i_0}( 1 - P(I_{-i} > i))\prod_{i=i_0+1}^{\infty}
				( 1 - \E(I_{0}^{1 + \epsilon})/i^{1+\epsilon}) > 0,
			\end{align*}
			where $i_0 \in \N$ is chosen such that 
			$\E(I_{0}^{1 + \epsilon})/i^{1+\epsilon} < 1$ for all $i > i_0$.

			For (ii), the argument is exactly the same. The only difference is that we no longer require positivity of $p_1$, and instead, we use $\P( I_0 \leq 1)$, which is positive by assumption.  
			
			We now consider part (a) of item (iii). The event $\{ O = j\}$ implies $I_{-i} = i+j$ for some $i \geq 0$. Recall that for this part we have assumed $\E(I_0^{2 + \epsilon})<\infty$ for some $\epsilon > 0$. As $O$ is a non-negative integer valued r.v., using Markov inequality we obtain  
			\begin{align*}
				\E(O) = & \sum_{n=1}^\infty 
				P(O \geq n) \\ 
				= & \sum_{n=1}^\infty \P( \cup_{i \in \Z^+} I_{-i} \geq n + i ) \\
				\leqslant & \sum_{n=1}^\infty \sum_{i = 0}^\infty \P(  I_{-i} \geq  n + i ) \\
				\leqslant &  \sum_{n=1}^\infty \sum_{i = 0}^\infty \E(I_{-i}^{2+\epsilon})/(n + i)^{2+\epsilon}\\
				= &  \E(I_0^{2+\epsilon})\sum_{n=1}^\infty \sum_{i = 0}^\infty 1/(n + i)^{2+\epsilon}\\
				= &   \E(I_0^{2+\epsilon})\sum_{m=1}^\infty (m+1)/m^{2+\epsilon} < \infty.
			\end{align*}
			The rearrangement in the last step is justified as we are dealing with non-negative terms only.
			
			Part (b) of item (iii) follows from similar argument. Since $\E(I_0^{4+\epsilon})< \infty$, 
			Markov inequality gives us
			\begin{align*}
				\E(O^2) = & \sum_{n=1}^\infty 
				P(O^2 \geq n) \\ 
				= & \sum_{n=1}^\infty \P( \cup_{i \in \Z^+} I^2_{-i} \geq n + i ) \\
				\leqslant & \sum_{n=1}^\infty \sum_{i = 0}^\infty \P(  I_{-i} \geq  \sqrt{n + i} ) \\
				\leqslant &  \sum_{n=1}^\infty \sum_{i = 0}^\infty \E(I_{-i}^{4+\epsilon})/(n + i)^{(4+\epsilon)/2}\\
				= &  \E(I_{0}^{4+\epsilon}) \sum_{n=1}^\infty \sum_{i = 0}^\infty 1/(n + i)^{(4+\epsilon)/2}\\
				= &   \E(I_0^{4+\epsilon}) \sum_{m=1}^\infty (m+1)/m^{(4+\epsilon)/2} < \infty.
			\end{align*}
			This completes the proof.
			
			Finally, for part (c) of item (iii)
			we have assumed that there exist $C_0, C_1 > 0$ such that 
			$$
			P(I_0 > n) \leq C_0 e^{(-C_1 n)} \text{ for all }n.
			$$
			Under this assumption, we obtain
			\begin{align}
				\label{eq:Overshoot_ExpTail}
				\P(O > n) = \P( \cup_{i \in \Z^+} I_{-i} > n + i ) \leqslant 
				\sum_{i = 0}^\infty \P(  I_{-i} > n + i ) \leqslant  \sum_{i = 0}^\infty C_0e^{-C_1(n + i)}
				\leqslant C_2e^{-C_3 n},
			\end{align}
			for some $C_2, C_3 > 0$. 
		\end{proof}
		\begin{remark}
			\label{rem:}
			We comment here that item (iii) of \Cref{lem:rumour-block} does not require the collection $\{ I_i : i \in \Z^-\}$ to  be i.i.d. An uniform moment bound or an uniform tail bound is enough to conclude item (iii), e.g., part (a) of item (iii) holds for a collection $\{ I_i : i \in \Z^-\}$ satisfying $\E(I_i^{2+\epsilon}) \leq M < \infty$ for all $i$.   
		\end{remark}

		For any subset $\A \subset \Z$
		we consider the family of r.v.'s $\{ I_i : i \in \A \cap \Z\}$ and corresponding to this family, we define the `right overshoot' r.v. and `left overshoot' r.v. towards the right side and left 
		side of $\A$  respectively as below:
		$$
		O^{\A, +} := \max \{ (i + I_i) - \sup(\A)  : i \in \A \cap \Z \} \text{ and }
		O^{\A, -} := \max \{ \inf(\A) - (i - I_i)  : i \in \A \cap \Z \}.
		$$
		Clearly, the r.v. $O^{\A,+}$ ($O^{\A,-}$) makes sense only if $\A \subseteq (-\infty, b]$ ($\A \subseteq [a,+\infty)$) for some $b  \in \Z$ ($a  \in \Z$). On the other hand, for any bounded set $\A \subset  \Z$, 
		both these two r.v.s are a.s. finite. We are now ready to prove \Cref{thm:phase}.
	}
	
	\begin{proof}[Proof of \Cref{thm:phase}]
		It is immediate that for the choice $p_1 = 0$, rumour must percolate. 
		Hence, to complete the proof of \Cref{thm:phase}, it 
		suffices to show \cref{eq:Tau_Exp}. Set ${\cal F}_0 = \{\emptyset, \Omega \}$ and for $n \geq 1$ we define ${\cal F}_n := \sigma( I_z : z \in \A_{n-1} )$.
		It is not difficult to see that the rumour propagation process 
		$\{\A_n : n \geq 0\}$ is adapted to the filtration $\{{\cal F}_n : n \in \N \}$.
		
		For $n \in \N$ let $\tA^+_n := \tA_n \cap \Z^+$ denote the set of active vertices at time point 
		$n$ on the positive side. Similarly, the set $\tA^-_n$ is defined. 
		Both the random sets $\tA^+_n$ and $\tA^-_n$ are ${\cal F}_{n}$ measurable.
		Given the $\sigma$-field ${\cal F}_{n}$, 
		we consider the overshoot r.v.'s $O^{\tA^+_n,+}$ and $O^{\tA^-_n,-}$. 
		By construction, for all $n \geqslant1$ we have $\tA^+_n \cap \tA^-_n = \emptyset$ a.s. Therefore, given the $\sigma$-field ${\cal F}_{n}$, 
		the overshoot r.v.'s $O^{\tA^+_n, +}$ and $O^{\tA^-_n, -}$ are supported on disjoint sets of influence r.v.'s, and for any two Borel sets $B_1, B_2$ we have 
		\begin{align}
			\label{eq:Ind_Overshoot}
			\P\bigl ( O^{\tA^+_n,+} \in B_1, O^{\tA^-_n, -} \in B_2 \mid {\cal F}_{n} \bigr ) 
			= \P\bigl ( O^{\tA^+_n,+} \in B_1 \mid {\cal F}_{n} \bigr )
			\P\bigl ( O^{\tA^-_n,-} \in B_2 \mid {\cal F}_{n} \bigr ).
		\end{align}
		Further, given ${\cal F}_{n}$ the conditional distributions of both the overshoot r.v.'s, $O^{\tA^+_n,+}$ and $O^{\tA^-_n,-}$, are 
		stochastically dominated by the overshoot r.v.\ $O$ as defined in \cref{def:J_Geometric_Tail} in the following sense: 
		\update{\begin{align}\label{eq:StochDom_Overshoot}
				& \P\bigl ( O^{\tA^+_n,+} = 0 \mid {\cal F}_{n} \bigr )\wedge 
				\P\bigl ( O^{\tA^-_n,-} = 0 \mid {\cal F}_{n} \bigr )\geqslant \P(O =  0)  \text{ and }\nonumber\\
				&  \P\bigl ( O^{\tA^+_n,+}  > m \mid {\cal F}_{n} \bigr )\vee 
				\P\bigl ( O^{\tA^-_n,-} > m \mid {\cal F}_{n} \bigr )\leqslant \P(O > m) \text{ for any }m \in \N,   
			\end{align}
			where $'\wedge'$ and $'\vee'$ denotes the minimum and maximum respectively.}
		
		Given $\tau > n$, the r.v.\ $\tau$ takes the value $n+1$ if and only if
		the event $\{ O^{\tA^+_n,+} = O^{\tA^-_n,-} = 0\}$ occurs. 
		Therefore, using \cref{eq:Ind_Overshoot} and \cref{eq:StochDom_Overshoot} we obtain  
		\update{\begin{align*}
				\P(\tau = n + 1 \mid \tau > n, {\cal F}_n) &  = \P(O^{\tA^+_n,+} = 0, O^{\tA^-_n,-} = 0 
				\mid \tau > n, {\cal F}_n) \nonumber \\ 
				&  = \P(O^{\tA^+_n,+} = 0 \mid \tau > n, {\cal F}_n)\P(O^{\tA^-_n,-} = 0 \mid \tau > n, {\cal F}_n) \\
				& \geqslant \P(O = 0)^2  > 0.
		\end{align*}} 
		This gives us  
		\update{$$
			P(\tau > n) \leqslant (1- \P(O = 0)^2)^n,
			$$} 
		and completes the proof. 	
	\end{proof} 
	
	\update{The above result allows us to obtain moment bounds for the
		size of the sub-critical rumour cluster.
		Let $M = \# \A_{\tau}$ denote the size of the rumour cluster. Under the assumption that $\E(I_0^{1 + \epsilon}) < \infty$ for some $\epsilon > 0$, we proved that $M$ is finite a.s. if $p_1 > 0$. 
		
		\begin{prop}
			\label{prop:CusterSizeMomentBound}
			\begin{itemize}
				\item[(i)] If $p_1 > 0$ and $\E(I_0^{2 + \epsilon}) < \infty$ for some $\epsilon > 0$, then $M$ has finite expectation. 
				\item[(ii)] If $p_1 > 0$ and $\E(I_0^{4 + \epsilon}) < \infty$ for some $\epsilon > 0$, then $\E(M^2) < \infty$.
				\item[(iii)] If $p_1 > 0$ and there exist $C_0, C_1> 0$ such that $\P(I_0 > n) < C_0 e^{(-C_1 n)}$ for all $n \in \N$, then $M$ also has exponentially decaying tail.
			\end{itemize}
		\end{prop}
		\begin{proof}
			Set $p_1 > 0$.
			Let $\# \A^+_{\tau} = M_+$ and $\#\A^-_{\tau} = M_-$. Clearly $M \leq M_+ + M_-$. From the translation invariance nature of our model it follows that both these two r.v.'s, 
			$M_+$ and $M_-$ are identically distributed. Therefore, it is enough to prove \Cref{prop:CusterSizeMomentBound} in terms of $M_+$. 
			
			Next, we observe that for any $n \geq 0$ given $\mathcal{F}_n$, the increment $\# \tA^+_{n +1}$ is stochastically dominated by the overshoot r.v. $O$. 
			Therefore, the r.v. $M_+$ is stochastically dominated by the random sum $\sum_{m = 1}^\tau O_m$ where $\{ O_m : m \in \N\}$ gives a family of i.i.d. copies of $O$. With the help of item (iii) of Lemma \ref{lem:rumour-block}, the proof of \Cref{prop:CusterSizeMomentBound} for $M_+$  readily follows from the following corollary.
		\end{proof}
		
		\begin{corollary}
			\label{lem:RandomSumMoment}
			Consider a collection of i.i.d. r.v.'s $\{X_i : i \geq 1\}$ and let $\eta$ be a non-negative integer valued r.v. Consider the random sum $Z := \sum_{i=1}^\eta X_i$ 
			\begin{itemize}
				\item[(i)] If $X_1$ and $\eta$ both have finite expectation then $Z$ has finite expectation.
				\item[(ii)] If $\E(X^2_{1}) < \infty$ and $\eta$ is of finite expectation, then $Z$ has finite second moment.
				
				\item[(iii)] If there exist $C_0, C_1 > 0$ such that $\P(X_1 > n)\vee \P(\eta > n) \leq C_0 e^{(-C_1 n)}$ for all $n \geq 1$, then $Z$ has exponentially decaying tail.
			\end{itemize}
		\end{corollary}
		
		The first and the second parts of the Corollary follow from the Wald's equation and the Blackwell-Girshick equation. 
		The third part is also a well known result, e.g., we refer the reader to Lemma 2.7 of \cite{RSS16A}.}
	
	\update{ Before ending this section below application of our methods for some related models.
		\begin{remark}
			\label{rem:Extensions}
			\begin{itemize}
				\item[(a)] Consider an i.i.d. collection of Bernoulli random variables $\{X_u : u \in \Z\}$ and assume that there is an individual at $u$, who can participate for rumour propagation, if and only it $X_u = 1$. This model was introduced in \cite{athreya2004} where the question of complete coverage has been  studied. This problem was also mentioned in Section 5 of \cite{GGJR2014} as a possible direction for extension. For this model under the assumption that $\E(I^{1+\epsilon}_0)$ is finite, our method applies as well and gives us that irrespective of the value of $p_1 \in [0,1]$, percolation won't happen. Further, the survival time of the subcritical rumour cluster decays exponentially too.  This follows from the observation that if $\P(I_0 \leq m) > 0$ for some $m \in \N$ then the same argument as in Lemma \ref{lem:rumour-block} gives us $\P( O \leq m) > 0$ as well. We note that for any $n \geq 1$, the occurrence of the event 
				$$
				E_n := \{ I_u < m \text{ for all }m \in \tA_n\}\cap \{ X_v = 0 \text{ for all }v\notin \A_n \text{ with }|v-w|\leq m \text{ for some }w \in \A_n \}
				$$
				implies rumour percolation stops, and at each step probability of this event is uniformly bounded below.    
				
				\item[(b)] Authors of \cite{athreya2004} analysed a more general model compared to the previous one where individuals are placed on $\Z$ as a $\{0,1\}$-valued time-homogeneous Markov chain. This model was mentioned in Section 5 of \cite{GGJR2014} as a future question to explore. As long as we have 
				$\P(X_1 = 0 \mid X_0 = 0)\wedge \P(X_1 = 0 \mid X_0 = 1) > 0$, our method applies in this set up also (under the assumption $\E(I^{1+\epsilon}_0)$ is finite) and gives the same set of conclusions.
			\end{itemize}    
		\end{remark}
	}
	
	\section{Speed of the rightmost vertex in rumour cluster in the supercritical phase}
	\label{sec:speedrttail}

	In what follows, we will continue to work in the supercritical phase, i.e.,
	we assume that  $p_1 = 0$. 
	\update{In addition, for the next result we assume that 
		\begin{align}
			\label{eq:AssumptionSecondMoment}
			\E(I_0^{2 + \epsilon}) < \infty\text{ for some }\epsilon > 0.
	\end{align}}
	
	Let $r_n = \max \A_n$ and $l_n = \min \A_n$ 
	respectively denote the rightmost and the leftmost vertex in the set $\A_n$.
	Clearly, for any $n \geqslant1$, both $r_n$ and $l_n$ are a.s. finite r.v.'s and 
	the rumour cluster is represented by the set $[l_n, r_n]\cap \Z$. Under the assumption $p_1 = 0$, we have
	$[-n, n] \subseteq [l_n, r_n]$  for all $n \geqslant1$ a.s.
	The main result of this section is the following: 
	\begin{theorem}
		\label{thm:rightmost-speed}
		\update{Fix any $\epsilon > 0$. Under the assumption that the radius of influence r.v. has $2+\epsilon$ moment finite, there exists a deterministic constant $\mu \geqslant 1$ such that we have
			$$
			\lim_{n \to \infty} \frac{r_n}{n} = \mu \quad  \text{ a.s.}
			$$ }
	\end{theorem}
	
	\update{To prove \Cref{thm:rightmost-speed}  in the following section
		we introduce a renewal construction for the rumour propagation process.}

	\subsection{A renewal process}
	\label{subsec:RenewalConstruction}
	
	\update{In this section, we describe a renewal construction for the rumour propagation process. In \cite{GGJR2014} the authors observed a renewal process based on the firework process on $\N$. For rumour propagation on $\Z$ our renewal process
		construction is along different lines. The main challenge here lies in the fact that the rumour
		can spread on the positive side as well as on the negative side of the integer line $\Z$ here, making it more complicated than the firework process.}
	
	We define a random time $\sigma$ as 
	\begin{align}
		\label{def:Sigma}
		\sigma=\inf \{ n \geqslant 1 : O^{(-\infty, - n ],+} \leq  n \}.
	\end{align}
	Firstly, from the model description, it follows that $l_n \leq - n$ for all $n \geq 1$. 
	\update{Next corollary shows that the r.v. $\sigma$ is finite a.s.
		
		\begin{corollary}
			The r.v. $\sigma$ defined as in (\ref{def:Sigma}) is finite a.s.
		\end{corollary}
		\begin{proof}
			For $n \geq 1$ we consider the event $A_n := \{ I_{-n} \geq 2n \text{ for all }u \in \tA^-_m \} $ and observe the following event inclusion relation 
			$$
			\{ \sigma = + \infty\} \subset \limsup_{n \to \infty} A_n . 
			$$
			Under the assumption that $\E(I_0^{1 +\epsilon})< \infty$ for some $\epsilon > 0$, using Markov inequality we obtain 
			$$
			\sum_{n = 1}^\infty \P(A_n) = \sum_{n = 1}^\infty \P(I_0 \geq 2n) \leq  \E(I_0^{1 +\epsilon}) \sum_{n = 1}^\infty \frac{1}{(2n)^{1 + \epsilon}} < \infty.
			$$
			Therefore, the proof follows from  Borel Cantelli lemma.    
	\end{proof}}
	
	After the $\sigma$-th step, active vertices on the negative side, 
	i.e., vertices in the set $\tA^-_n$ for any $n > \sigma$, cannot 
	spread the rumour to inactive vertices on the positive side and vertices in the set 
	$\tA^+_n$ only can spread the rumour to the positive side.
	We observe that the distribution of $\sigma$ depends only on radii of influences r.v.'s
	on the negative part of the integer line and hence,  for any $z \in \N$
	the r.v.\ $I_z$ is independent of $\sigma$.

	We now define the following sequence of random steps: 
	\begin{align}
		\label{def:TauStep}
		& \tau_0 = \sigma \text{ and for }j\geqslant1 \nonumber \\
		& \tau_j = \inf \{ n > \tau_{j-1} : \# \tA^+_n  = 1 \}= \update{\inf \{ n > \tau_{j-1} :  r_n - r_{n - 1}  = 1 \}}.
	\end{align}
	We need to show that for any $j \geqslant1$, the r.v.\ $\tau_j$ is a.s.\ finite. 
	This is done in part (i) of Proposition \ref{prop:Tau_Tail}.
	We recall that the notation $r_n$ denotes the location of the rightmost 
	vertex of the rumour cluster at time $n$. 
	\begin{prop}
		\label{prop:Tau_Tail}
		\update{   \begin{itemize}
				\item[(i)]  There exist $C_0, C_1 > 0$ such that for all $n \geq1$ we have
				$$
				P(\tau_{j+1}-\tau_j > n \mid \tau_1, \cdots , \tau_{j-1} ) 
				\leqslant C_0 e^{(-C_1 n)} ; 
				$$
				\item[(ii)] Under the assumption (\ref{eq:AssumptionSecondMoment})  the increment r.v. $( r_{\tau_{j+1}}- r_{\tau_j} ) $ has finite expectation.
		\end{itemize} }   
	\end{prop}
	\begin{proof}
		Note that for the rumour propagation process starting from the origin, 
		the process $\{(\tA^+_{n}, \tA^-_{n}) : n \geq 1\}$ is Markov. 
		\update{ By definition of the random step $\sigma$, for any $n > \sigma$ 
			vertices in $\tA^-_{n}$ cannot spread the rumour to the positive side. Therefore, for any $n > \sigma = \tau_0$ we have
			\begin{align}
				\label{eq:Dominate_ActiveSet}
				\P \bigl ( \#\tA^+_{n+1}  = 1 \mid (\tA^-_{n}, \tA^+_{n}), (\sigma < n) \bigr ) & \geqslant  \P( O = 1) > 0, 
			\end{align}
			where $O$ is defined as in (\ref{def:J_Geometric_Tail}).
			
			This implies that after the $\sigma$-th step at every step, there is 
			a fixed positive lower bound for the probability of occurrence of a $\tau$ step.   
			Hence, \ref{eq:Dominate_ActiveSet} proves part
			(i) of \Cref{prop:Tau_Tail}.
			
			Regarding (ii), we observe that 
			\begin{equation}
				\label{eq:RightmostRandomSum}   
				r_{\tau_{j+1}}- r_{\tau_j} = \sum_{m = \tau_j + 1}^{\tau_{j+1}} \#\tA^+_{m}.
			\end{equation}
			Consider a collection of i.i.d. copies of the overshoot r.v. $O$ given as $\{ O_i : i \geq 1\}$ and let $\overline{\tau} := \inf \{ i \geq 1 : O_i = 1\}$. It follows that the increment r.v. $r_{\tau_{j+1}}- r_{\tau_j}$ is stochastically dominated by the random sum $\sum_{i=1}^{\overline{\tau}} O_i$. Finally, Corollary \ref{lem:RandomSumMoment} ensures under the assumption (\ref{eq:AssumptionSecondMoment}), i.e., $\E(I^{(2+\epsilon)}_0) < \infty$ for some $\epsilon > 0$, the random sum $\sum_{i=1}^{\overline{\tau}} O_i$ has finite expectation and therefore, the increment r.v. $( r_{\tau_{j+1}}- r_{\tau_j} ) $ has finite mean too. } 
	\end{proof}

	The next result shows that the above sequence of random steps gives a sequence of renewal steps with i.i.d. increments 
	for the rumour propagation process. 
	\begin{prop}
		\label{prop:Tau_IID}
		$\{ (\tau_{j+1} - \tau_j, r_{\tau_{j+1}} - r_{\tau_{j}}) : j \geqslant1\}$ 
		gives a collection of i.i.d. random vectors taking values in $\N \times \N$.
	\end{prop}
	\begin{proof}
		Independence of increment vectors between successive renewals follows as they 
		are supported on disjoint sets of radius of influence random variables. 
		We present a brief description of the identical distribution of the increment random 
		vector below.
		
		Consider a rumour propagation process only on $\Z^+$ starting with the origin as the 
		only active vertex at time point $0$. Let $\A^+_n$ and $\tA^+_n$ denote 
		the corresponding analog of sets $\A_n$ and $\tA_n$ when the sets $\A_n$ and $\tA_n$ are now restricted to $\Z^+$ only. We consider the evolution of this process till 
		the random time $\gamma$ such that we have $\# \tA^+_\gamma = 1$. Let $r^{+}_\gamma$ denote the position of the rightmost active vertex at time $\gamma$.
		The increment random vectors between successive renewals 
		have the same distribution as $(\gamma, r^{+}_\gamma)$, i.e.,
		for all $j \geqslant 1$ we have
		$$(\tau_{j+1} - \tau_j, r_{\tau_{j+1}} - r_{\tau_{j}})
		\stackrel{{\cal L}}{=} (\gamma, r^{+}_\gamma),$$
		\update{where the notation $\stackrel{{\cal L}}{=}$ denotes equality in distribution.} This completes the proof.
		
	\end{proof}

	\subsection{Proof of \Cref{thm:rightmost-speed}}
	
	In this section we complete
	the proof of Theorem \ref{thm:rightmost-speed}. \update{In the next corollary, using our renewal construction, we prove Theorem \ref{thm:rightmost-speed} and provide a precise 
		understanding of the limit too.} 
	\begin{corollary}
		\label{cor:Understanding_Limit}
		We have 
		\begin{align}\label{eq:speed_mu}
			\lim_{n \to \infty} r_n/ n = \mu =\frac{\E(r_{\tau_2} - r_{\tau_1})}{\E({\tau_2} - {\tau_1})}
			\text{a.s.}
		\end{align}
	\end{corollary}
	\begin{proof}
		From \Cref{prop:Tau_IID}, we can write the right hand side of \cref{eq:speed_mu} as: 
		$$\lim_{k \to \infty} \dfrac{\frac{1}{k} \left(\sum_{j=1}^k r_{\tau_{j+1}}-r_{\tau_j}\right)}{\frac{1}{k} \left(\sum_{j=1}^k \tau_{j+1}-\tau_{j}\right)}.$$
		This reduces to give, 
		\begin{align}\label{eq:exp_ratio_r_tau}
			\frac{\E(r_{\tau_2} - r_{\tau_1})}{\E({\tau_2} - {\tau_1})} = \lim_{k \to \infty} \frac{r_{\tau_{k+1}}-r_{\tau_1}}{\tau_{k+1}-\tau_{1}}.\end{align}
		Let $n$ be large enough such that $\tau_1 < n$. Let $k(n) \in \N$ be such that $\tau_{k(n)} \leqslant n < \tau_{k(n)+1}$. From now on we shall just use $k$ to denote $k(n)$. Then, 
		$r_{\tau_k} - r_{\tau_1} \leqslant r_n - r_{\tau_1} < r_{\tau_{k+1}} - r_{\tau_1}.$ Since we have 
		\begin{align*}
			\lim_{k \to \infty} \frac{r_{\tau_{k+1}}-r_{\tau_1}}{\tau_{k+1}-\tau_{1}} = \lim_{k \to \infty} \frac{r_{\tau_{k}}-r_{\tau_1}}{\tau_{k}-\tau_{1}}, 
		\end{align*}
		hence, 
		\begin{align}\label{eq:ratio_r_tau_n}
			\lim_{n \to \infty} \frac{r_{n}-r_{\tau_1}}{n-\tau_{1}} = \lim_{k \to \infty} \frac{r_{\tau_{k+1}}-r_{\tau_1}}{\tau_{k+1}-\tau_{1}}.
		\end{align}
		So, from \cref{eq:exp_ratio_r_tau} and \cref{eq:ratio_r_tau_n} we get, 
		\begin{align*}
			\frac{\E(r_{\tau_2} - r_{\tau_1})}{\E({\tau_2} - {\tau_1})} = \lim_{n \to \infty} \frac{r_{n}-r_{\tau_1}}{n-\tau_{1}} = \lim_{n \to \infty} \frac{r_{n}}{n}. 
		\end{align*}
		
	\end{proof} 
	\subsection{A central limit theorem for the position of rightmost vertex}
	
	\update{In this section, we prove a central limit theorem for the position of the rightmost vertex in a super-critical rumour percolation cluster, motivated by similar work by \cite{kuczek1989central} on the right edge of super-critical oriented percolation. For this part of the result we assume that $\E(I_0)^{4+\epsilon} < \infty$ for some $\epsilon > 0$. Under this assumption by item (iii) of \Cref{lem:rumour-block} we have that the overshoot r.v. $O$ has finite second order moment. Further, \Cref{lem:RandomSumMoment} and the same argument as in part (ii) of Proposition \ref{prop:Tau_Tail} together give us that the increment r.v. $(r_{\tau_{j+1}} - r_{\tau_{j}})$ has finite second order moment as well.        
	}

	The main result of the section is the following: 
	\begin{theorem}
		\label{thm:CLT}
		\update{Under the assumption that $\E(I_0)^{4+\epsilon} < \infty$ for some $\epsilon >0$ and $p_1 = 0$, we have 
			\begin{equation}\label{eq:clt_rn}
				\frac{r_n - n \mu}{\sqrt{n}} \overset{\mathcal{L}}{\longrightarrow} N(0, \psi^2),
			\end{equation} 
			where the notation $\overset{\mathcal{L}}{\longrightarrow}$ denotes convergence in distribution.}
	\end{theorem}
	\begin{proof} Choose $n$ large enough such that $\tau_1 < n$. Let $k(n) \in \N$ be such that $\tau_{k(n)} \leqslant n < \tau_{k(n)+1}$. From now on we shall just use $k$ to denote $k(n)$. Then, 
		$r_{\tau_k} \leqslant r_n < r_{\tau_{k+1}}.$ 
		So we also have, $$\frac{r_{\tau_k} - n \mu}{\sqrt{n}} \leqslant\frac{r_n - n \mu}{\sqrt{n}} < \frac{r_{\tau_{k+1}} - n \mu}{\sqrt{n}}.$$ 
		
		From \Cref{prop:Tau_IID} we know that $\frac{\tau_{j+1} - \tau_j}{\sqrt{n}}$ converges in probability to $0$ as $n \to \infty$, for all $j\geqslant1$, also that $\frac{r_{\tau_{k+1}}-r_{\tau_k}}{\sqrt{n}} \overset{P}{\longrightarrow} 0$.
		Hence, we shall have $\frac{r_{\tau_{k+1}}-r_{n}}{\sqrt{n}} \overset{P}{\longrightarrow} 0$. So, it remains to show that, 
		\begin{equation*}
			\frac{r_{\tau_{k+1}} - n \mu}{\sqrt{n}} \overset{\cal L}{\longrightarrow} N(0, \psi^2). 
		\end{equation*}
		From \cite[Lemma 2]{siegmund1975time} as stated in \cite[Corollary 1]{kuczek1989central}, we know that in our notation 
		\begin{equation*}
			\frac{r_{\tau_{k+1}} - n\mu}{\sqrt{n}} \overset{\cal L}{\longrightarrow} N(0, \psi^2). 
		\end{equation*}
		Here, $\psi$ is a positive real number. On incorporating \Cref{cor:Understanding_Limit}, we have \cref{eq:clt_rn}. 
	\end{proof}

	\section{A rumour model with a reactivation mechanism}\label{sec:reactivation}

	In this section, we will introduce and analyse a  rumour model with a reactivation mechanism \update{where, due to reactivation, a vertex which received the rumour at a past time point, 
		may spread it again in future using an independent radius of influence random variable. The motivation is that, on the event of spreading a gossip over a network, an enthusiastic individual may contribute at multiple time points. As we have mentioned before, this concept of reactivation is influenced by the idea of recurring public opinion and its dissemination, considered for example in \cite{xu2020dynamic}. It is further motivated by the concept of reinfection in case of disease spread models. 
		Below, we describe our model.}
	
	We start with a family of i.i.d.\ radii of influence r.v.'s.
	$\{ I^n_v : v \in \Z,  n \in \Z^+ \}$  and a family of i.i.d.\ Bernoulli r.v.'s
	$\{ B^n_v : v \in \Z,  n \in \Z^+ \}$ with parameter $\update{p_2} \in (0,1)$. 
	A vertex $u \in \Z$, which has newly heard the rumour at time $n$, uses the r.v.\ $ I^n_u$ to spread the rumour. In addition to it, a vertex 
	$u^{\prime} \in \Z$ which heard it at some time before $ n$, becomes active again and spreads the rumour using $ I^n_{u^\prime}$ at time $n$ if and only if $B^n_{u^\prime} = 1$. 
	In other words, for a vertex $u \in \Z$, which heard the rumour at a previous time point, 
	the Bernoulli r.v. $B^n_u $ indicates whether $n$ is an activation time or not.   
	\update{For this section, we would assume that the radius of influence r.v.  $I^0_0$ has exponentially decaying tail. Our results, however, would still hold true without assuming this, but with assumptions on higher moments only.}
	We are now ready to define this rumour propagation with the reactivation model formally.  
	
	For any interval $I$ with $I \cap \Z \neq \emptyset$, we denote the rumour percolation process with reactivation starting from an 
	initial active set $I \cap \Z$ as $\{(\A^{\text{R},I}_n, \tA^{\text{R},I}_n) : n \geqslant 0\}$
	where $\A^{\text{R},I}_n$ denotes the set of vertices that have heard the rumour by time $n$ and 
	$\tA^{\text{R},I}_n$ denotes the set of vertices that are ready to spread the rumour at time $n$. 
	The random set $\tA^{\text{R},I}_n$ definitely contains vertices that 
	have newly heard the rumour at time $n$. More importantly, it may contain  an additional
	set of reactivated vertices that heard the rumour in the past.
	We define this model inductively as follows:
	\begin{align*}
		\A^{\text{R},I}_0 & := \tA^{\text{R},I}_0 = I \cap \Z \text{ and }\\ 
		\tA^{\text{R},I}_1 & := \{ u \in \Z \setminus \A^{\text{R},I}_0 : |u - w| \leqslant I^0_w \text{ for some }w\in 
		\tA^{\text{R},I}_0\} \cup 
		\{ u \in \A^{\text{R},I}_0 : B^{1}_u = 1 \}.\\
		\A^{\text{R},I}_1 & := \A^{\text{R},I}_0 \cup \tA^{\text{R},I}_1. 
	\end{align*}
	The set $\A^{\text{R},I}_1$ represents the rumour cluster  at time $1$, i.e., the set of vertices 
	that have heard the rumour by time $1$. We recall that for the conventional rumour percolation model, there is no reactivation, and 
	we have $\tA_m \cap  \tA_n = \emptyset$ for all $m, n \in \Z^+$
	with $m \neq n$. This no longer holds for the reactivation model. In fact, a vertex $ u \in \tA^{\text{R},I}_0$ 
	belongs to $\tA^{\text{R},I}_1$ as well iff we have $B^1_u = 1$. 
	The set of vertices ready to spread the rumour at time $2$ is given by 
	\begin{align*}
		\tA^{\text{R},I}_2 & := \{ u \in \Z \setminus \A^{\text{R},I}_1 : |u - w| \leqslant I^1_w \text{ for some }w \in  
		\tA^{\text{R},I}_1\} \cup \{ u \in \A^{\text{R},I}_1 : B^2_u = 1 \}.\\
		\A^{\text{R},I}_2 & := \A^{\text{R},I}_1 \cup \tA^{\text{R},I}_2. 
	\end{align*}
	The set $\A^{\text{R},I}_2$ represents the rumour cluster  at time $2$, i.e., the set of vertices that have heard the rumour by time $2$. More generally, for $n \geqslant 2$ given $(\A^{\text{R},I}_{n-1}, \tA^{\text{R},I}_{n-1} )$, we define the 
	sets $\tA^{\text{R},I}_{n} $ and $\A^{\text{R},I}_{n} $ respectively as:
	\begin{align*}
		\tA^{\text{R},I}_{n}  & := \{ u \in \Z \setminus \A^{\text{R},I}_{n-1} : |u - w| \leqslant I^{n-1}_w \text{ for some }w \in  
		\tA^{\text{R},I}_{n-1}\} \cup \{ u \in \A^{\text{R},I}_{n-1} : B^{n}_u  = 1 \}.\\
		\A^{\text{R},I}_{n}  & := \A^{\text{R},I}_{n-1} \cup \tA^{\text{R},I}_{n}.
	\end{align*}
	$\A^{\text{R}}_{n}$ represents the random set of vertices that have heard the rumour by time 
	$n$ and therefore, it represents the rumour cluster at time $n$. 
	It follows that for any vertex which heard the rumour in the past, the time gaps 
	between successive reactivations follow i.i.d. geometric distributions with parameter \update{$p_2$},
	which represents the laziness parameter of the reactivation clock.
	
	With a slight abuse of notations, for an interval $I$ of the form $I\cap \Z = \{u\}$, we denote the corresponding rumour propagation process simply as $\{ (\A^{\text{R},u}_n,\tA^{\text{R},u}_n ) : n \geq 0\}$.
	In particular,  for simplicity of notation the process  $\{ (\A^{\text{R},0}_n,\tA^{\text{R},0}_n ) : n \geq 0\}$ is simply denoted as 
	$\{ (\A^{\text{R}}_n,\tA^{\text{R}}_n ) : n \geq 0\}$. 
	It is not difficult to see that, similar to the earlier model, 
	for any $u \in \Z$ and $n \geqslant 1$, the random set $\A^{\text{R},u}_{n}$ is a.s. finite and of 
	the form $[a, b]\cap \Z$, for some $-\infty < a < b <\infty$.
	
	The rumour percolation event for this model is defined  
	as in \Cref{def:RumourPercolation}. We observe that, unlike the earlier model, 
	the event 
	$$
	\A^{\text{R}}_n = \A^{\text{R}}_{n + 1}\text{ for some }n \geqslant 1 ,
	$$
	no longer implies non-occurrence of percolation. In fact, even if the radius r.v. $I^0_0$ takes the value zero with positive probability, it is not hard to see that for all $p_2>0$ 
	this reactivation model percolates. \update{In this section we will assume that the i.i.d. family of radius of influence r.v.'s is such that the following two conditions are satisfied:
		\begin{itemize}
			\item[(a)] There exist $C_0, C_1 > 0$ such that 
			$\P(I^0_0 > n) \leq C_0 e^{(-C_1 n)}$ for all $n \geq 1$ and 
			\item[(b)] $\P(I^0_0 = 0) = p_1 > 0$.
	\end{itemize}}

	Let $l^{\text{R}}_n$ and $r^{\text{R}}_n$ denote 
	the leftmost and rightmost vertex in the set $\A^{\text{R}}_n$ respectively. \update{Under the above mentioned assumptions on the i.i.d. family $\{I^n_z : n \in \Z^+, z \in \Z \}$  we show that
		almost surely the random quantity $r_n^R / n$ has a deterministic limit. We comment that one can proceed along the same line of arguments as in the previous section and under the assumption of existence of sufficient higher order moments of radius of influence r.v. Theorem \ref{thm:restarst_rightmost-speed} can be proved. However, in this article we chose to work with the exponentially decaying radii of influence r.v.'s.  
	}
	\begin{theorem}
		\label{thm:restarst_rightmost-speed}
		There exists some deterministic constant $\mu^\prime > 0$ such that almost surely we have 
		$$
		\lim_{n \to \infty} r_n^R / n = \mu^\prime . 
		$$
	\end{theorem}
	\update{The next section and the rest of the paper is dedicated to the proof of Theorem \ref{thm:rightmost-speed}. Though the proof technique is the same, because of the re-activation the renewal construction requires non-trivial modification. }

	\subsection{Renewal process for reactivation rumour process}
	\label{subsec:Renewal_restart}

	We need to modify our renewal construction in \Cref{subsec:RenewalConstruction}  in a non-trivial manner such that it takes care of the re-activation mechanism. 
	Recall that we have started with two independent collections $\{I^n_z : n \in \Z^+, z \in \Z \} $ and 
	$\{ B^n_z : n \in \Z^+, z \in \Z \}$ of i.i.d. random variables. Before proceeding further, we define an one-sided version of our rumour 
	propagation model, which we will use in our renewal event construction.
	
	For $u  \in \Z$ we define an one-sided version of rumour propagation 
	such that starting from $u$ the rumour propagates towards the right side of $u$ only, i.e., on the set $[u, \infty)\cap \Z$. 
	Set $\A^{R,+,u}_0 = \tA^{R,+,u}_0 = \{u\}$ and define
	\begin{align*}
		\tA^{R,+,u}_1 &  := \{ v  \in \Z : v > u  : |v - u| \leq I^0_u \} \cup\{ v \in \A^{R,+,u}_0 : B^1_v = 1 \} \text{ and }\\
		\A^{R,+,u}_1 &  := \A^{R,+,u}_0 \cup \tA^{R,+,u}_1.    
	\end{align*}
	Given $(\A^{R,+,u}_n, \tA^{R,+,u}_n)$ we define 
	\begin{align*}
		\tA^{R,+,u}_{n+1} &  := \{ v  \in \Z : v > \max(\A^{R,+,u}_n)  : |v - w| \leq I^n_w \text{ for some }
		w \in \tA^{R,+,u}_n \} \cup\{ v \in \A^{R,+,u}_n : B^{n+1}_v = 1 \}, \\
		\A^{R,+,u}_{n+1} &  := \A^{R,+,u}_n \cup \tA^{R,+,u}_{n+1}.    
	\end{align*}
	As earlier, the set $\A^{R,+,u}_{n+1}$ represents the set of vertices that have heard the rumour by time $n+1$ of this one-sided rumour propagation with re-activation model. On the other hand, the set $\tA^{R,+,u}_{n+1}$ represents the set of vertices that are ready to spread the rumour at time $n+1$. For any interval $I$ with $I \cap \Z \neq \emptyset$, we define the processes $\{(\A^{R,+,I}_{n+1}, \tA^{R,+,I}_{n+1}) : n \geq 1\}$ similarly.
	
	We observe that  for any two intervals $I_1$ and $I_2$, the two independent families $\{ I^n_z : n \in \Z^+, z \in \Z\}$ and 
	$\{ B^n_z : n \in \Z^+, z \in \Z\}$ allow us to couple together all the processes mentioned below:
	\begin{align}\label{eq:Processes}
		& \{(\A^{R,I_1}_n, \tA^{R,I_1}_n) : n \geq 0\}, 
		\{(\A^{R,I_2}_n, \tA^{R,I_2}_n) : n \geq 0\} \text{ and } \nonumber \\
		& \{(\A^{R,+,I_1}_n, \tA^{R,+,I_1}_n) : n \geq 0\}, 
		\{(\A^{R,+,I_2}_n, \tA^{R,+,I_2}_n) : n \geq 0\}.
	\end{align}

	Next, for $u \in \Z$ we consider the two processes 
	$\{(\A^{R,(-\infty,u)}_n, \tA^{R,(-\infty,u)}_n) : n \geq 0\}$ and 
	$\{(\A^{R,+,u}_n, \tA^{R,+,u}_n) : n \geq 0\}$ coupled together. For $u \in \Z$, the event `$\text{Dom}(u)$' 
	means that propagation of the rumour towards the 
	right side of $u$ starting from an initial set of active vertices $(-\infty,u)\cap \Z = (-\infty,u-1]\cap \Z$ 
	remains dominated throughout by the one-sided rumour propagation process starting 
	from a singleton active set $\{u\}$. Mathematically, we define this event as 
	\begin{align}\label{def:Dom_Event}
		\text{Dom}(u) := \{ \max(\A^{R,+,(-\infty,u)}_n) \leq \max(\A^{R,+,u}_n) 
		\text{ for all } n \geq 1\}.
	\end{align}
	It is not difficult to observe that for every $n \geq 1$ we have $\A^{R,+,(-\infty,u)}_n = 
	\A^{R,(-\infty,u)}_n$ a.s.
	
	We are now ready to define our renewal steps for the process $\{(\A^{R}_n, \tA^{R}_n) : n \geq 0\}$.
	Given $(\A^{R}_n, \tA^{R}_n)$ we say that the renewal event occurs at the $n$-th step if 
	the event $\text{Dom}(r^R_n) \equiv \text{Dom}(\max(\A^{R}_n))$ occurs.
	The sequence of successive renewal steps is defined below:
	
	Set $\tau^R_0 = -1$. For $j \geq 1$, the $j$-th renewal step for this 
	rumour propagation with reactivation is defined as 
	\begin{align} 
		\label{def:Tau_Reactivation}
		\tau^R_j & := \inf\{ n \in \Z^+, n > \tau^R_{j-1} : \text{Dom}(r^R_n) \text{ occurs}\}.
	\end{align}
	We need to show that the above definition makes sense in the sense that 
	for any $j \geqslant 1$, the r.v.\ $\tau^R_j$ is a.s.\ finite. 
	This is done in Proposition \ref{prop:Tau_Tail_Reactivate}.
	We define two filtrations $\{{\cal F}^R_n : n \geq 0\}$ and $\{{\cal G}^R_n : n \geq 0\}$
	such that ${\cal F}^R_0 = {\cal G}^R_0 = \emptyset$ and for $n \geq 1$
	\begin{align}
		\label{def:filtration}
		{\cal F}^R_n & := \sigma \bigl ( ( I^m_z, B^{m+1}_{z} ) : z \in (-\infty, r^R_{n-1}]\cap \Z , 
		0 \leq m \leq n- 1 \bigr )  \text{ and } \nonumber \\
		{\cal G}^R_n & := \sigma  \bigl  ( (I^m_z, B^{m+1}_{z}, \mathbf{1}_{\text{Dom}(r^R_m)} ) : z \in (-\infty, r^R_{n-1}]\cap \Z ,  0 \leq m \leq n- 1 \bigr ).
	\end{align}
	We observe that both the processes, $\{ (\A^{R,(-\infty,0)}_n, \tA^{R,(-\infty,0)}_n) : n \geq 0\}$ 
	and $\{(\tA^{R}_n, \tA^{R}_n) : n \geq 0\}$, are adapted to the filtration $\{{\cal F}^R_n : n \geq 0\}$. 
	Further, for any $j \geq 1$ the random step $\tau^R_{j}$ is a stopping time 
	w.r.t. the filtration $\{{\cal G}^R_n : n \geq 0\}$. This gives us another filtration 
	$\{{\cal S}^R_j := {\cal G}^R_{\tau^R_j} : j \geq 0\}$ to work with such that the random 
	variable  $\tau^R_j$ is ${\cal S}^R_j$ measurable for all $j \geq 0$.
	\begin{prop} 
		\label{prop:Tau_Tail_Reactivate}
		There exist $C_0, C_1 > 0$ which do not depend  on $j$ such that 
		for any $j \geqslant 0$ we have
		$$
		\P(\tau^R_{j+1} - \tau^R_{j} > n \mid {\cal S}^R_j) 
		\leqslant C_0 e^{(-C_1 n)} \text{ for all }n \in \N .
		$$
	\end{prop}
	We first prove Proposition \ref{prop:Tau_Tail_Reactivate} for $j = 0$ and we will it prove for general 
	$j > 0$ later. In order to do that, we need the following lemma 
	which proves that given ${\cal F}^R_n$, the probability that the $n$-th step is a renewal step or a $\tau^R$ step has a strictly positive uniform lower bound.
	
	\begin{lemma}
		\label{lem:Reactivate_Tau_LowerBd}
		For any $n \geq 0$ there exists $\tilde{p} > 0$ not depending on $n$ such that 
		$$
		\P(n-\text{th step is a renewal step} \mid {\cal F}^R_{n}) = \P( \text{Dom}(r^R_n) \mid {\cal F}^R_{n}) \geq \tilde{p}.
		$$
	\end{lemma}
	\update{To prove \Cref{lem:Reactivate_Tau_LowerBd} we need the following corollary 
		involving a family of random variables with uniformly bounded $2 + \epsilon$ moments for some $\epsilon > 0$.
		\begin{corollary}
			\label{cor:E2_Event}
			Consider a family of i.i.d. random variables $\{ Y_n : n \in \N\}$  with finite $2 + \epsilon$ moment for some $\epsilon > 0$. In addition, we assume that 
			$\P(Y_1 = 0) > 0$. Then for any $\gamma > 0$, we have 
			$$
			\P( Y_n \leq n\gamma \text{ for all }n \geq 1) > 0. 
			$$
		\end{corollary}
		\begin{proof}
			We have assumed that $\E(Y_1)^{2 +\epsilon} = M  < \infty$.
			Define the event 
			$$A_n := \{ Y_{n+j} > (n + j)\gamma \text{ for some }j \geq 1\}$$ and we obtain  
			\begin{align}
				\label{eq:React_Cor}
				\P(A_n) = & \P( \cup_{j=1}^\infty (Y_{n+j} > (n+j)\gamma) ) \nonumber \\
				\leq  & \sum_{j=1}^\infty \P(  Y_{n+j} > (n+j)\gamma ) \nonumber \nonumber \\
				\leq  & \sum_{j=1}^\infty \E(  Y_{n+j})^{2+\epsilon} / ((n+j)\gamma )^{2+\epsilon} \nonumber \\
				\leq  & \sum_{j=1}^\infty \E(  Y_{n+j})^{2+\epsilon} / ((n+j)\gamma )^{2+\epsilon} \nonumber \\
				=   & M \gamma^{-(2+\epsilon)} \sum_{j=1}^\infty 1 / (n+j)^{(2+\epsilon)}. 
			\end{align}
			We define $
			N := \sup \{ n \geq 1 : A_n \text{ occurs}\}$,
			i.e., the last time the event $A_n$ has occurred. Equation (\ref{eq:React_Cor}) gives us 
			$$
			\sum_{n = 1}^\infty \P(A_n) \leq  
			M  \sum_{n = 1}^\infty \sum_{j=1}^\infty 1 / ((n+j)\gamma )^{2+\epsilon} = 
			M \sum_{m = 1}^\infty  (m-1)/ (m\gamma )^{2+\epsilon} < \infty,
			$$
			implying that the r.v. $N$ is finite a.s. Choose $K \in \N$ such that $\P(N \leq K) \geq 1/2$. On the other hand, we obtain $$
			\P(\cap_{n=1}^{K} (Y_n \leq n \gamma)) \geq
			\P(\cap_{n=1}^{K} (Y_n = 0)) = \P(Y_1 =0)^K > 0,
			$$
			by our assumption. Further, the event $\{N \leq K\}$ depends only on the collection $\{Y_{K+1}, Y_{K+2},\cdots  \}$ and therefore, independent of the event $\cap_{n=1}^{K} (Y_n \leq n \gamma)$. Finally, our proof follows from the event inclusion relation given below:   $$
			\{Y_n \leq n \gamma \text{ for all }n\} \supseteq  (\cap_{n=1}^{K} (Y_n = 0))\cap (N \leq K).
			$$
			This completes the proof.
	\end{proof}}

	It is important to observe that in \Cref{cor:E2_Event} we don't require $\{ Y_n : n \in \N\}$ to be 
	an i.i.d. collection. We are ready to prove \Cref{lem:Reactivate_Tau_LowerBd} now.
	
	\begin{proof}[\bf Proof of \Cref{lem:Reactivate_Tau_LowerBd}:] 
		We prove \Cref{lem:Reactivate_Tau_LowerBd} for $n = 0$. For general $n \geq 0$ the argument is 
		exactly the same. We couple a positive drift random walk with the evolution of the rightmost vertex of the 
		rumour cluster as follows. Set $X_0 = 0$ and for $n \geq 1$ let
		\begin{align*} 
			X_{n} := \begin{cases}
				X_{n-1} + 1 & \text{ if }B^{n}_{r^R_{n-1}} = 1, I^{n}_{r^R_{n-1}} \geq 1\\
				X_{n-1} & \text{ otherwise}.
			\end{cases}
		\end{align*}
		By construction $\{ X_n : n \geq 0\}$ is a positive drift random walk with i.i.d. 
		increments such that we have $\lim_{n \to \infty} X_n/n = \theta $ almost surely where 
		\begin{align}
			\label{def:Theta}   
			\theta = \P(B^0_{0} = 1, I^0_{0} \geq 1) > 0 .
		\end{align}
		Therefore, there exists $\eta_1 > 0$ such that 
		$$
		\P( X_n \geq  \frac{3}{4} n \theta   \text{ for all } n \in \Z^+) = \eta_1.
		$$
		Our construction ensures that $X_n \leq r_n^R$ for all $n \geq 0$ a.s.
		
		Next, we consider the sequence of `right' overshoot random variables given by 
		$\{ O^{(-\infty, -1],+}_{n}: n \in \Z^+ \}$, 
		where $O^{(-\infty, -1], +}_{n}$ represents the amount of overshoot towards on $\Z^+$ due to 
		the family $\{ I_z^n : z \in (-\infty, -1] \cap  \Z \}$. The translation invariance nature of our model ensures that the collection $\{ O^{(-\infty, -1],+}_{n}: n \in \Z^+ \}$ gives a sequence of i.i.d. copies of the overshoot r.v. $O$, where $O$ is defined as in \cref{def:J_Geometric_Tail}. 
		Define  the event $E$ as
		$$
		E := \{ O_{n}^{(-\infty, -1],+} \leq \frac{1}{4} n \theta  \text{ for all }n \in \Z^+\}.
		$$
		\Cref{lem:rumour-block} together with \Cref{cor:E2_Event} ensure that $\P(E) \geq \eta_2$ for some $\eta_2 > 0 $.
		
		The occurrence of the event $\{ X_n \geq \frac{3}{4} n \theta  \text{ for all } n \in \Z^+ \}\cap E$ implies the occurrence of the renewal step at the very first step. Further, these two events are supported on disjoint sets of random variables and they are therefore independent.
		This gives us
		$$
		\P(0-\text{th step is a renewal step }\mid {\cal F}^R_0) \geq \eta_1 \eta_2 > 0,
		$$
		and completes the proof. 
	\end{proof}
	
	\update{Additionally to prove \Cref{prop:Tau_Tail_Reactivate} we need \cite[Corollary 2.7]{RSS23}, which
		is a general result giving a sufficient condition for exponential tail decay of a random sum of random variables 
		which need not be i.i.d. We will use this condition repeatedly in this section. For completeness, we mentioned the details below. We start with a notion 
		of `uniform exponential decay' and `strong  exponential decay' for a family of r.v.'s. 
		
		\begin{definition}
			\label{def:Exp_Tail}
			We say that a family of r.v.'s $\{X_i : i \geqslant 1\}$ has uniform exponential tail 
			decay if uniformly for all $i \geqslant 1$ there exist constants $C_0, C_1 > 0$ such that 
			\begin{align}
				\label{eq:UnifExpDecay_1}
				\P \bigl ( |X_i| > n ) \bigr )
				\leqslant C_0 e^{(-C_1 n)} \text{ for all }n \in \N,
			\end{align}
			where $C_0, C_1$ do not depend on $i$.
			
			We say that a family of r.v.'s $\{X_i : i \geqslant 1\}$ has strong  
			exponential tail decay if uniformly for all $i \geqslant 1$ there exist constants $C_0, C_1 > 0$ such that 
			\begin{align}
				\label{eq:StrongExpDecay_1}
				\P \bigl ( X_i > n \mid (X_{i-1},\cdots, X_1) \bigr )
				\leqslant C_0 e^{(-C_1 n)} \text{ for all }n \in \N.
			\end{align}
		\end{definition} 
		
		\begin{lemma}  
			\label{lem:RandomSumExpTail} 
			Consider a family of r.v.'s $\{ X_i: i\geqslant 1\}$ with strong  exponential 
			tail decay and a positive integer-valued r.v. $Y$ with an exponentially decaying tail.
			Then there exist $C_0, C_1 > 0$ such that for all $ n \in \N$ we have
			$$
			\P( \sum_{i=1}^Y |X_i| > n) \leqslant  C_0 e^{-C_1 n}.
			$$
		\end{lemma}

		For a proof of \Cref{lem:RandomSumExpTail} we refer the reader to Corollary 2.7 of \cite{RSS23}. 
		It is important to observe that \Cref{lem:RandomSumExpTail} does not 
		assume independence between the r.v. $\sigma$ and the family $\{ X_n : n \in \N \}$ too. 
		Now we use \Cref{lem:Reactivate_Tau_LowerBd}, \Cref{lem:RandomSumExpTail} and proceed to prove  
		\Cref{prop:Tau_Tail_Reactivate} for $j = 0$.
	}
	
	\begin{proof}[\bf Proof of Proposition \ref{prop:Tau_Tail_Reactivate} for $j = 0$:]
		\update{It is important to observe that the 
			occurrence of a renewal event affects the distribution of radius of influence r.v.'s and Bernoulli activation  r.v.'s till the infinite future. Therefore, the argument is much more involved here than that of Proposition \ref{prop:Tau_Tail}.  } 
		We first check for the occurrence of the renewal 
		event at the very first step and Lemma \ref{lem:Reactivate_Tau_LowerBd}
		ensures that we have a strictly positive uniform lower bound for probability of the occurrence of the renewal event. \update{But given the renewal  step has not occurred at the first step gives us certain information about the future evolution of the process and  
			prohibits us to apply the same bound (for probability of occurrence of renewal event) at the next step. We need to wait for some random amount of time to take care of this information. }
		
		For $u \in \Z$ we define the random variable $\beta^R(u)$ as 
		\begin{align}
			\label{def:Beta_Reactivate}
			\beta^R(u) := \inf\{ n \geq 1 : \max(\A^{R,(-\infty,u)}_{n}) > \max(\A^{R, +, u}_{n})\}. 
		\end{align}
		\update{We observe that $\max\left(\A^{R,(-\infty,u)}_{n}\right) = \max\left(\A^{R,(-\infty,u-1]}_{n}\right)$ for all $n$ a.s. and the r.v. $\beta^R(u)$ can take the  value infinity with positive probability.} 
		For the rumour propagation process $\{ (\A^R_n, \tA^R_n) : n \geq 0\}$
		starting from   the origin, we consider the r.v. $\beta^R_1 = \beta^R(0)$ and observe that 
		the  event $\{\beta^R_1 = \infty\}$ is exactly same as the event that the $0$-th step is a renewal step. \update{Given the event that renewal has not occurred at the $0$-th step, which is same as the event $\{ \beta^R_1 < \infty \}$, we have some information till the next $\beta^R_1$ many steps and we can again test for the occurrence of the renewal event post that. On the event $\{ \beta^R_1 < \infty \}$  we run the rumour propagation process for the next $\beta^R_1 + 1$ many steps and  thereafter we consider the r.v. $\beta^R_2 := \beta^R(r^R_{\beta^R_1 + 1})$ and so on.}
		It suffices to show that the family $\{\beta^R_k : k \geq 1\}$ exhibits a
		strong exponential tail decay type behaviour as defined in \Cref{def:Exp_Tail}.
		We prove it for $k = 1$ and for general $k \geq 1$ the argument is the same. 
		
		For showing the exponential tail for $\P( n < \beta^R_1 < \infty )$, we write it as, 
		\begin{align}\label{eq:beta1-exp-s1}
			\P( n < \beta^R_1 < \infty ) \leq \sum_{l = 1}^\infty  \P( \beta^R_1 = n + l , r_{n+l}^R \geq \frac{3}{4} (n + l) \theta  )
			+ \sum_{l = 1}^\infty \P( r_{n+ l}^R <  \frac{3}{4} (n + l) \theta  ) ,
		\end{align}
		where $\theta$ is as in \cref{def:Theta}.
		The second term in \cref{eq:beta1-exp-s1} will be dealt with later using a concentration bound. 
		We now consider the first term in \cref{eq:beta1-exp-s1}. 
		\update{We observe that on the event 
			$$\{ r_{n+l}^R \geq  \frac{3}{4} (n + l)\theta , \beta^R_1 = n + l \}$$
			we must have 
			\begin{align*}
				& \max(\A^{R,(-\infty,-1]}_{m}) \leq \max(\A^{R,+, 0}_{m}) \text{ for all } 0 \leq m \leq n+l-1  \text{ as well as }  \\  
				& \max(\A^{R,(-\infty,-1]}_{n+l}) > \max(\A^{R,+, 0}_{n+l}) \geq  \frac{3}{4} (n + l)\theta.
			\end{align*}
			In order to have $ \max(\A^{R,(-\infty,-1]}_{n+l}) > \frac{3}{4}(n+l)\theta$, there must exist a random variable 
			$I^{n+l}_z$ for some $z \in (-\infty, -1]\cap \Z$ with $$
			I^{n+l}_z + z > \frac{3}{4} (n + l)\theta.
			$$
			This implies that the first term inside the sum in the r.h.s of (\ref{eq:beta1-exp-s1}), viz., $\P( \beta^R_1 = n + l , r_{n+l}^R \geq  \frac{3}{4} (n + l)\theta  )$ for every $l \geq 1$  is bounded by 
			$\P( O^{(-\infty,-1],+}_{n + l} >  \frac{3}{4}(n + l) \theta )$. Hence, we have 
			\begin{align}
				\label{eq:beta1-exp-s3}
				\sum_{l = 1}^\infty  \P( \beta^R_1 = n + l , r_{n+l}^R \geq  \frac{3}{4} (n + l)\theta  ) 
				\leqslant \sum_{l = 1}^\infty  \P( O^{(-\infty,-1],+}_{n + l} > \frac{3}{4} (n + l) \theta  )
				\leqslant \tilde{C}^\prime_0 e^{-\tilde{C}^\prime_1 n },   
			\end{align}
			for some $\tilde{C}^\prime_0, \tilde{C}^\prime_1 > 0$. 
			The last inequality \cref{eq:beta1-exp-s3} follows from \Cref{lem:rumour-block}.
		}

		We now consider the second term in \cref{eq:beta1-exp-s1}.  
		Recall that $\lim_{n \to \infty} X_n / n = \theta$ almost surely.
		From \cite[eq.\ (7)]{hagerup1990guided}, and the fact that $X_n \leq r^R_n$ almost surely for all $n\geq 0$, we have 
		\begin{align}\label{eq:beta1-exp-s5}
			\P( r_{n+ l }^R <  \frac{3}{4} (n + l ) \theta  ) \leqslant 
			\P( X_{n+ l } <  \frac{3}{4} (n + l ) \theta  ) \leqslant B^\prime_0 e^{-B^\prime_1 (n+l) }, 
		\end{align}
		for some $B^\prime_0, B^\prime_1 > 0$. 
		Plugging \cref{eq:beta1-exp-s3} and \cref{eq:beta1-exp-s5} into \cref{eq:beta1-exp-s1} and taking the sum in the second term, we get
		\begin{align*}
			\P( n < \beta^R_1 < \infty ) \leq C^\prime_0 e^{-C^\prime_1 n },
		\end{align*}
		for some $C^\prime_0, C^\prime_1 > 0$. A similar argument shows that the probability 
		$ \P( n < \beta^R_j < \infty \mid \beta^R_{j-1}, \cdots , \beta^R_{1}) $ decays exponentially with uniform decay constants. This ensures that the family $\{\beta_k : k \geq 1\}$ exhibits a strong  exponential tail decay. \update{The total number of steps required for the occurrence of the first renewal event is dominated by geometric sum of $\beta^R_j$'s with success probability $\tilde{p}$ where $\tilde{p}$ is as in  \Cref{lem:Reactivate_Tau_LowerBd}. Therefore,  \Cref{lem:RandomSumMoment} completes the proof for $j = 0$.}
	\end{proof}
	
	We will prove \Cref{prop:Tau_Tail_Reactivate} for general $j \geq 1$ through a sequence of lemmas. 
	We need to be more careful as given that a renewal step has occurred affects distribution of radius of influence r.v.'s and Bernoulli activation r.v.s 
	till the {\it infinite future}. To deal with that for $u \in \Z$ and $n \geq 1$ 
	we define a truncated version of the domination  event as follows:
	\begin{align}
		\label{def:Dom_Trunc_Event}  
		\text{Dom}^{(n)}(u) & := \{ \max(\A^{R,(-\infty,u)}_m) \leq \max(\A^{R^+,u}_m) 
		\text{ for all } 1 \leq m \leq n\},
	\end{align}
	\update{which means that domination of the rumour propagation continues till the next $n$ steps only. Next lemma allows us to decouple a renewal event as joint occurrence of two independent events.}
	\begin{lemma}
		\label{lem:EventDecomposition_1}
		For any $w \in \N$ we have the following equality of events:
		\begin{align}\label{eq:event_decom}
			\{\tau^R_1 = l, r^R_l =  w \}  &  = \Bigl [ \{ r^R_l =  w \} \cap 
			\bigl [ \cap_{j=0}^{l-1}
			\bigl ( \{ \beta^R(r^R_j) \leq l - j \}  \bigr ) \bigr ] \Bigr ] 
			\bigcap (\text{Dom}(w)),
		\end{align} 
		\update{where the two events on the r.h.s. are independent of each other.}
	\end{lemma}
	\begin{proof}
		By definition of the step $\tau_1$, we have 
		\begin{align*}
			\{\tau^R_1 = l, r^R_l = w\} = & \{ r^R_l = w, \text{Dom}(w)\} 
			\bigcap \bigl [ \cap_{j=0}^{l-1} ( \text{renewal does not occur at } r^R_j) \bigr]\\ 
			= & \{ r^R_l = w, \text{Dom}(w)\} \bigcap 
			\bigl[ \cap_{j=0}^{l-1}  (\beta^R(r^R_j) < \infty ) \bigr].
		\end{align*}
		In order to complete the proof, for any $0 \leq j \leq l-1$ we need to show that
		\begin{align}\label{eq:dom_eq}
			\{ r^R_l = w, \text{Dom}(w), \beta^R(r^R_j) < \infty \} =
			\{ r^R_l = w, \text{Dom}(w),\beta^R(r^R_j) \leq l - j \}.  
		\end{align}
		
		We prove it for $j = 0$ and the argument is exactly the same for general $j \geq 1$.
		On the event 
		$$
		\{r^R_l = w, \text{Dom}(w), \beta^R(r^R_0) = \beta(0) > l  \},
		$$
		the truncated event $\text{Dom}^l(0)$ must have occurred.
		Otherwise, $\beta^R(0)$ must be smaller than $l$. 
		Further, the event $\text{Dom}(r^R_l) \cap \text{Dom}^l(0)$ implies the occurrence 
		of the event $\text{Dom}(0)$. 
		Hence, on the event $\{ r^R_l = w, \text{Dom}(w), \beta^R(0) > l \}$, 
		the r.v. $\beta^R(0)$ must take the value $+\infty$. This implies that 
		$ \{ r^R_l = w, \text{Dom}(w), l < \beta^R(0) < +\infty \} = \emptyset $
		and completes the proof of \cref{eq:dom_eq} for $j=0$.
		Finally, we observe that the two events in the r.h.s. of \cref{eq:event_decom} depend on disjoint sets of r.v.s as the event $$ \{ r^R_l = w \} \cap \bigl [ \cap_{j=0}^{l-1}
		\bigl ( \{ \beta^R(r^R_j) \leq l - j \}  \bigr ) \bigr ] $$
		is ${\cal F}^R_l$ measurable where ${\cal F}^R_l$ is defined as in \cref{def:filtration}. Therefore, they are independent and this completes the proof.
	\end{proof}
	We would like to have a similar result for the $\tau_j$-th step for $ j \geq 2$. 
	For $m_1, m_2 \in \N$ with $m_1 < m_2$ we consider the event
	$$
	\{ \tau^R_1 = m_1, r^R_{m_1} = w_1, \tau^R_2 = m_2, r^R_{m_2} = w_2 \}.
	$$
	The main difficulty is that, both the events $\text{Dom}(w_1)$  and $\text{Dom}(w_2)$
	depend on the infinite future. To deal with this, we observe the following equality of events
	\begin{align*}
		\{r^R_{m_1} = w_1, \text{Dom}(w_1) , r^R_{m_2} = w_2, \text{Dom}(w_2)\} 
		& = \{r^R_{m_1} = w_1, \text{Dom}^{(m_2 - m_1)}(w_1) , r^R_{m_2} = w_2, \text{Dom}(w_2)\}.
	\end{align*}
	\update{
		Using the above decomposition, the same argument as in \Cref{lem:EventDecomposition_1} allows us to express the event $
		\{ \tau^R_1 = m_1, \tau^R_2 = j_2, r^R_{m_1} = w_1, r^R_{m_2} = w_2 \}$ as
		\begin{align}
			\label{eq:EqualityEvent_3}
			& \Bigl [ \{  r^R_j \text{ is not a }\tau^R \text{ step for all } 1 \leq j < m_1, r_{m_1} = w_1, 
			\text{Dom}^{ (m_2 - m_1)}(w_1) \}  \cap  \nonumber\\
			& \qquad \{ r^R_j \text{ is not a }\tau^R \text{ step for any }  m_1 < j < m_2, 
			r^R_{m_2} = w_2\} \Bigr ]  \bigcap \text{Dom}(w_2). 
		\end{align}
		As observed earlier, other than the event $\text{Dom}(w_2)$ in 
		\cref{eq:EqualityEvent_3}, rest of the events are ${\cal F}^R_{m_2}$ measurable.
		In fact, \cref{eq:EqualityEvent_3} can be further strengthened 
		for any $j \geq 1$ as described in the following corollary.}

	\begin{corollary}
		\label{cor:IncrDistEquality_1}
		For any $j \geq 1$ we obtain the following equality of events
		\begin{align}
			\label{eq:EventEquality_4}
			& \bigcap_{l=1}^j \{\tau^R_l = m_l, r^R_{m_l} = w_l \} =  \nonumber \\
			& \qquad \Bigl [ \bigcap_{l=1}^j  \{ r^R_{m_l} = w_l \}\cap 
			\bigl \{ \cap_{n = m_{l-1} + 1}^{m_l - 1} 
			( \beta^R(r^R_n) \leq m_l - n)\bigr \}\bigcap (\cap_{l=1}^{j-1}
			\text{Dom}^{ (m_{(l+1)} - m_l)}(w_l) )\Bigr ]\bigcap \text{Dom}(w_j),
		\end{align} 
		and other than the event $\text{Dom}(w_j)$, rest of the other events 
		in the r.h.s. of \cref{eq:EventEquality_4} are ${\cal F}^R_{m_l}$ measurable. 
	\end{corollary}
	The argument is same as that of \Cref{lem:EventDecomposition_1} and therefore, we skip the details. 
	For simplicity of notation, we denote the  ${\cal F}^R_{m_l}$ measurable part of the event in the r.h.s. of \cref{eq:EventEquality_4} as 
	$$
	E(w_1, m_1, \cdots , w_j, m_j) := 
	\Bigl [ \bigcap_{l=1}^j  \{ r^R_{m_l} = w_l \}\cap \bigl \{ \cap_{n = m_{l-1} + 1}^{m_l - 1} 
	( \beta^R(r_n) \leq m_l - n)\bigr \}  \bigcap (\cap_{l=1}^{j-1}
	\text{Dom}^{( m_{l+1} - m_l)}(w_l) )\Bigr ].
	$$
	It is also important to observe that the event $\text{Dom}(w_j)$ is actually independent of the event
	$E(w_1, m_1, \cdots , w_j, m_j)$. We need one more lemma before we proceed 
	to prove \Cref{prop:Tau_Tail_Reactivate} for $j \geq 1$.
	\begin{lemma}
		\label{lem:Reactivate_j}
		Fix any $j \geq 1$. Given $(r^R_{\tau^R_1},\tau^R_1, \cdots , r^R_{\tau^R_j}, \tau_j^R) = 
		(w_1,m_1, \cdots , w_j, m_j)$, for any $l \geq 1$ and $z \in \Z$ 
		there exist $C_0, C_1 > 0$, which do not 
		depend on $j, l, z$ such that  
		\begin{align*}
			\P \bigl( I^{\tau_j + l}_z > n \mid (r^R_{\tau^R_1},\tau^R_1, \cdots , r^R_{\tau^R_j}, 
			\tau^R_j) = (w_1, m_1, \cdots , w_j, m_j) \bigr )  \leq C_0 e^{(-C_1 n)}.  
		\end{align*}
	\end{lemma}
	\begin{proof}
		Using the event equality in (\ref{eq:EventEquality_4}), and from conditional probability we obtain
		{ \small
			\begin{align}\label{eq:Reactivate_j-s1}
				& \P \bigl( I^{\tau^R_j + l}_z > n \mid (r_{\tau^R_1},\tau^R_1, \cdots , 
				r^R_{\tau^R_j}, \tau^R_j) = (w_1, m_1, \cdots , w_j, m_j) \bigr ) \nonumber \\
				& = \P \bigl( I^{\tau^R_j + l}_z > n \cap E(w_1, m_1, \cdots , w_j, m_j) \cap \text{Dom}(w_j) \bigr )/
				\P(E(w_1, m_1, \cdots , w_j, m_j) \cap \text{Dom}(w_j)). 
		\end{align}}
		The event $(I^{\tau^R_j + l}_z > n) \cap \text{Dom}(w_j)$ is 
		independent of the event $E(w_1, m_1, \cdots , w_j, m_j)$. Earlier we observed that the event $\text{Dom}(w_j)$ is independent of the event $E(w_1, m_1, \cdots , w_j, m_j)$ as well.
		Hence, \cref{eq:Reactivate_j-s1} reduces to  
		\begin{align}
			& \P \bigl( I^{\tau_j + l}_z > n \mid (r^R_{\tau^R_1},\tau^R_1, 
			\cdots , r^R_{\tau^R_j}, \tau^R_j) = (w_1, m_1, \cdots , w_j, m_j) \bigr ) \nonumber \\
			= & \P \bigl( (I^{\tau^R_j + l}_z > n) \cap \text{Dom}(w_j) \bigr )/
			\P(\text{Dom}(w_j)) \nonumber \\
			\leqslant & \P \bigl( I^{\tau^R_j + l}_z > n \bigr ) \P(\text{Dom}(0))^{-1} \nonumber\\
			\leq & C_0 e^{(-C_1 n)}, \nonumber
		\end{align}
		for some $C_0, C_1 > 0$. The penultimate inequality uses the translation invariant nature of our model.
	\end{proof}
	Now we are ready to prove \Cref{prop:Tau_Tail_Reactivate} for $j \geq 1$. 
	
	\begin{proof}[\bf Proof of \Cref{prop:Tau_Tail_Reactivate} for $j \geq 1$] 
		We begin our proof by first fixing $j \geq 1$. Then, given the value of 
		$$
		(r^R_{\tau^R_1},
		\tau^R_1, \cdots , r^R_{\tau^R_j}, \tau^R_j) = (w_1,m_1, \cdots , w_j, m_j),
		$$
		for any $l \geq 1$ we can write the conditional probability as: 
		\begin{align}\label{eq:tau_tail:s1}
			& \P \bigl( \text{Dom}(r^R_{\tau^R_j + l}) \mid (r^R_{\tau^R_1},
			\tau^R_1, \cdots , r^R_{\tau^R_j}, \tau^R_j) = (w_1, m_1, \cdots , w_j, m_j) \bigr ) \nonumber \\   
			= & \P \bigl ( \text{Dom}(r^R_{m_j + l}) \cap ( E(w_1,m_1,\cdots , w_j,m_j)
			\cap \text{Dom}(w_j) )\bigr )/ \P(E(w_1,m_1, \cdots , w_j, m_j)\cap \text{Dom}(w_j) ). 
		\end{align}
		From \Cref{cor:IncrDistEquality_1}, the RHS of \cref{eq:tau_tail:s1} is equal to: 
		\begin{align}\label{eq:tau_tail:s2}
			\P \bigl ( \text{Dom}(r^R_{m_j + l}) \cap ( E(w_1, m_1, \cdots , w_j, m_j)
			\cap \text{Dom}^{(l)}(w_j) )\bigr )/ 
			\P(E(w_1, m_1, \cdots , w_j, m_j)\cap \text{Dom}(w_j) ).  
		\end{align}
		Further, as commented earlier, occurrence of the event $\text{Dom}(r^R_{m_j + l})$
		depends on the collection $\{ (I^n_w, B^n_w) : w \in \Z, n \geq m_j + l \}$ and it is 
		independent of every ${\cal F}^R_{m_j + l}$ measurable event. Therefore, we can simplify the expression in (\ref{eq:tau_tail:s2}) as, 
		\begin{align}\label{eq:tau_tail:s3}
			\P(\text{Dom}(r^R_{m_j + l}))\Bigl ( \P \bigl ( E(w_1, m_1, \cdots , w_j, m_j)
			\cap \text{Dom}^{(l)}(w_j) \bigr )/ 
			\P(E(w_1, m_1, \cdots , w_j, m_j)\cap \text{Dom}(w_j) ) \Bigr) . 
		\end{align}
		\update{From the definition of a renewal step it follows that rumour propagation through vertices in the set $( -\infty, r^R_{\tau^R_j})\cap \Z$ has to be dominated  and therefore, the translation invariance  nature of our model ensures that $\P(\text{Dom}(r^R_{m_j + l}))$ in (\ref{eq:tau_tail:s3}) is at least as large as  $\P(\text{Dom}(0))$. 
			On using the fact that 
			$$
			\bigl (E(w_1, m_1, \cdots , w_j, m_j)\cap \text{Dom}(w_j) \bigr )\subset 
			\bigl ( E(w_1, m_1, \cdots , w_j, m_j)\cap \text{Dom}^{(l)}(w_j) \bigr ), 
			$$ 
			the expression in (\ref{eq:tau_tail:s3}) is lower bounded by $\P(\text{Dom}(0))$.
		}

		This shows that given $(r^R_{\tau^R_1},\tau^R_1, \cdots , r^R_{\tau^R_j}, \tau^R_j)
		= (w_1,m_1, \cdots , w_j, m_j)$, for any $l \geq 1$ the probability of the event $\text{Dom}(r^R_{\tau^R_j + l})$ has a strictly positive uniform lower bound. 
		
		\update{The same argument of \Cref{prop:Tau_Tail_Reactivate} for $j=0$ together with  \Cref{cor:E2_Event} give us that, given 
			$(r^R_{\tau^R_1}, \tau^R_1, \cdots , r^R_{\tau^R_j}, \tau^R_j) 
			= (w_1,m_1, \cdots , w_j, m_j)$, it suffices to show that the collection $\{ O^{(-\infty, r^R_{\tau^R_j + l}- 1],+}_{\tau^R_j + l + n} 
			: l \geq 1 \}$ forms an uniform exponentially decaying family in the sense of  \Cref{def:Exp_Tail}.} 
		
		We recall that for each $n \geq 1$ part (iii) of \Cref{lem:rumour-block} holds  as long as the collection $\{ I_z^n : z \in \Z \}$ forms an uniform exponentially decaying family.  
		In the present set up, \Cref{lem:Reactivate_j} ensures that given 
		$(r^R_{\tau^R_1},\tau^R_1, \cdots , r^R_{\tau_j}, \tau^R_j) = (w_1,m_1, \cdots , w_j, m_j)$ for any $n \in \N$, the family $\{ I^{\tau^R_j + n}_z : z \in \Z \}$
		forms an uniform exponentially decaying family. As a result, for each $l \geq 1$ 
		the overshoot r.v. 
		$O^{(-\infty, r^R_{\tau^R_j + l}),+}_{\tau^R_j + l + n}$ has exponentially decaying tail such that decay constants do not depend on $l$. This completes the proof.
	\end{proof}
	Next, we show that increments of the rightmost process observed between successive 
	renewals form a strong uniform exponentially decaying family. 
	Together with the translation invariance nature of our model, \Cref{cor:E2_Event} and \Cref{lem:Reactivate_j} give us the following remark. 
	\begin{remark}
		\label{rem:Reactivate_j}
		\begin{itemize}
			\item[(i)] For any $z \in \Z$ and for any $m \in \N$ we consider the conditional overshoot r.v. 
			$O^{(-\infty, z],+}_m \mid \text{Dom}(0)$ and it is not difficult to see that there exist positive constants $C_0, C_1 $ not depending on $z, m$ such that we have
			$$
			\P \bigl ( O^{(-\infty, z],+}_{m} > n \mid \text{Dom}(0) \bigr ) \leq C_0 e^{(-C_1 n)}
			\text{ for all }n.
			$$
			This follows from the observation that 
			\begin{align*}
				\P \bigl ( O^{(-\infty, z],+}_{m} > n \mid \text{Dom}(0) \bigr ) 
				\leq \P ( O^{(-\infty, z],+}_{m} > n) / \P( \text{Dom}(0) ) =
				\P ( O^{(-\infty, 0],+} > n) / \P( \text{Dom}(0) ),
			\end{align*}
			where the last equality follows from the translation invariance nature of our model.
			
			\item[(ii)] Repeated applications of the above argument gives us that, given that the event $\text{Dom}(0)$ has occurred, for any $\Z$-valued sequence $\{ z_n : n \in \N \}$ the collection $\{O^{(-\infty, z_n],+}_{ n} \mid  : n \geq 1 \}$ forms an uniform exponentially decaying family. In fact, the same argument actually gives us that $\{O^{(-\infty, z_n],+}_{ n}   : n \geq 1 \}$ forms a strong exponentially decaying family. 
		\end{itemize}
	\end{remark}
	\begin{lemma} 
		\label{lem:Tau_Tail_Reactivate_1}
		There exist $C_0, C_1 > 0$ which do not depend on $j \geq 0$ such that 
		for any $j \geqslant 0$ we have
		$$
		\P(r^R_{\tau^R_{j+1}} - r^R_{\tau^R_{j}} > n \mid {\cal S}^R_j) \leqslant C_0 e^{(-C_1 n)}  \text{ for all }n \in \N.
		$$
	\end{lemma}
	\begin{proof}
		We have used similar arguments in this paper.
		Therefore, we only provide a sketch here. We observe that 
		\begin{align}
			\label{eq:Reactivate_Rtmost_Incr}  
			(r^R_{\tau^R_{j+1}} - r^R_{\tau^R_{j}}) = \sum_{m = 1}^{\tau^R_{j+1} - \tau^R_{j}} 
			(r^R_{\tau^R_{j} + m} - r^R_{\tau^R_{j} + m - 1}) .
		\end{align}
		Note that given ${\cal S}^R_j$, the increment r.v. 
		$(r^R_{\tau^R_{j} + m} - r^R_{\tau^R_{j} + m - 1})$ is  stochastically dominated by the overshoot r.v. 
		$O^{(-\infty,r^R_{\tau^R_{j} + m - 1}],+}_{\tau^R_{j} + m} $. Because of \Cref{rem:Reactivate_j}, \Cref{prop:Tau_Tail_Reactivate} together with Corollary \ref{lem:RandomSumMoment}
		give us that the random sum in (\ref{eq:Reactivate_Rtmost_Incr}) decays exponentially and this completes the proof.
	\end{proof}
	
	\begin{prop} 
		\label{prop:Tau_IID_Reactivate} 
		$\{ (r^R_{\tau^R_{j+1}} - r^R_{\tau^R_{j}}, \tau^R_{j+1} - \tau^R_j) : j \geqslant 1\}$
		gives a collection of i.i.d. random vectors taking values in $\N \times \N$ whose 
		distribution does not depend on the starting point of the rumour propagation process. 
	\end{prop}  
	\begin{proof}
		Unlike the proof of \Cref{prop:Tau_IID}, we need to be careful to argue about the independence of increment vectors between successive renewals as given that the renewal event has occurred 
		affects the distribution of the rumour propagation process till infinite time. 
		We first argue that the increment random vectors are identically distributed and 
		present a brief description of the identical distribution.

		The one-sided rumour propagation process  on $\Z^+$ starting from the origin is denoted
		by 
		$$
		\{ (\A^{R, +,0}_n, \tA^{R, +,0}_n) : n \geq 0\}.
		$$
		Given that the event $\text{Dom}(0)$ has occurred, we consider 
		the evolution of the conditional process 
		$$
		\{ (\A^{R,+,0}_n, \tA^{R, +,0}_n) : n \geq 0\} \mid \text{Dom}(0)
		$$ 
		till the  time that the event $\text{Dom}(r^{R, +,0}_n)$ occurs. 
		Let $\nu$ denote the first time that the event $\text{Dom}(r^{R, +,0}_n)$ has occurred. 
		The definition of our renewal event ensures that 
		$$
		(\tau^R_{j+1} - \tau^R_j, r^R_{\tau^R_{j+1}} - r^R_{\tau^R_{j}})  
		\stackrel{d}{=} (\nu, \max(\A^{R, +,0}_\nu)) \text{ for all }j \geqslant1.
		$$
		
		Next, we use the fact that the increment random vectors between 
		successive renewals are identically distributed to show that the increments are independently distributed as well. Recall that the random vector 
		$(\tau^R_{j}, r^R_{\tau^R_{j}})$ is ${\cal S}^R_{j}$ measurable where 
		${\cal S}^R_{j}$ is defined as in \cref{def:filtration}. 
		Fix $ m \geq 1 $ and Borel subsets $ B_2, \dotsc, B_{m+1} $ of $ \N \times \N$. Let
		$ I_{j + 1} ( B_{j +1} ) $ be the indicator random variable of the event 
		$$
		(\tau^R_{j+1} - \tau^R_j, r^R_{\tau^R_{j+1}} - r^R_{\tau^R_{j}})  \in B_{j+1} .
		$$
		Then, we have
		\begin{align*}
			& \P \bigl(  (\tau^R_{j+1} - \tau^R_j, r^R_{\tau^R_{j+1}} - r^R_{\tau^R_{j}}) 
			\in B_{j + 1} \text{ for } j =  1, \dotsc, m \bigr )
			= \E (  \prod_{ j = 1}^{m } I_{j + 1} ( B_{j + 1} )) \\
			&  = \E \Bigl( \E  \bigl( \prod_{ j = 1}^{m } I_{j +1} ( B_{j +1} )
			\mid {\cal S}^R_{m} \bigl) \Bigr) = 
			\E \Bigl( \prod_{ j = 1}^{m-1 } I_{j + 1} ( B_{j + 1} )
			\E  \bigl(  I_{m+1} ( B_{m+1} ) \mid {\cal S}^R_{m} \bigl) \Bigr)
		\end{align*}
		as the random variables $ I_{j+1} ( B_{j +1} )  $ are measurable w.r.t. $ {\cal S}^R_{m} $
		for $ j =  1, \dotsc, m-1 $. 
		
		By the earlier discussion, we have that the conditional distribution of
		$ (\tau^R_{m +1} - \tau^R_m, r^R_{\tau^R_{m+1}} - r^R_{\tau^R_{m}})$ 
		given $ {\cal S}^R_{m}$ is given by $(\nu, \max(\A^{R, +,0}_\nu))$. 
		Therefore, we have
		\begin{align*}
			& \P (  (\tau^R_{j+1} - \tau^R_j, r^R_{\tau^R_{j+1}} - r^R_{\tau^R_{j}}) 
			\in B_{j + 1 } \text{ for } j =  1, \dotsc, m )\\
			& =  \E \Bigl( \prod_{ j = 1}^{m-1 } I_{j + 1} ( B_{j + 1} )
			\E  \bigl(  I_{m+1} ( B_{m+1} ) \mid {\cal S}^R_{m} \bigl) \Bigr) \\
			& = \P \bigl(  (\nu, \max(\A^{R, +,0}_\nu)) \in B_{m+1} \bigr ) 
			\E \Bigl( \prod_{ j = 1}^{m - 1 } I_{j + 1} ( B_{j + 1} ) \Bigr).
		\end{align*}
		Now, induction on $m$ completes the proof.
	\end{proof}
	
	\subsection{Proof of \cref{thm:restarst_rightmost-speed}}
	\label{subsec:restarst_rightmost-speed}
	
	\update{The same argument as in \Cref{thm:rightmost-speed} gives us that 
		\begin{align*}
			\lim_{n \to \infty} r^R_n/n = \lim_{j \to \infty} r^R_{\tau^R_j}/\tau^R_j  = \mu^\prime \text{ almost surely}.
	\end{align*}}

	\section*{Compliance with Ethical Standards}
	
	We would like to declare that there are no conflicts of interest, either financial or non-financial, in the preparation of the manuscript. No funding was received to assist with the preparation of the manuscript. The authors are solely responsible for the correctness of the statements provided in the manuscript. 
	
	We have not used any data, materials, or codes in our work. We have not conducted any experiments or dealt with subjects for which ethics approvals or consents are required. Both of the authors have contributed equally to the work, and we don't require the consent of any other individual or authority for the publication of this manuscript. 
	
	\section*{Acknowledgement}
	
	The authors would like to thank the referees and the editors for their valuable comments, which has helped the authors in significantly improving their results.

\end{document}